\documentclass[a4paper, 11pt]{article}
\usepackage[left=25mm, right=25mm, top=25mm, bottom=36mm]{geometry} 

\usepackage{amsmath}
\usepackage{amssymb}
\usepackage{amsthm}
\usepackage{amsfonts}
\usepackage[alphabetic,nobysame]{amsrefs}

\usepackage[USenglish]{babel}
\usepackage{enumitem}
\usepackage{xcolor}
\usepackage[colorlinks=true, linkcolor=blue, citecolor=red, filecolor=teal, urlcolor=teal]{hyperref}

\newcommand{\doiref}[1]{\href{https://doi.org/#1}{\textcolor{teal}{#1}}}
\newcommand{\mquad}{\mkern-18mu}

\newcommand{\coleq}{\mathrel{\mathop:}=}

\newcommand{\DD}{\mathrm D}
\newcommand{\dd}{\mathrm d}
\newcommand{\ee}{\mathrm e}
\renewcommand{\div}{\operatorname{div}}

\newtheorem{theorem}{Theorem}[section]
\newtheorem{definition}[theorem]{Definition}

\newtheorem{lemma}[theorem]{Lemma}
\newtheorem{proposition}[theorem]{Proposition}
\newtheorem{corollary}[theorem]{Corollary}
\newtheorem{remark}[theorem]{Remark}

\begin{document}
	\title{Exponential equilibration of an electro--energy--reaction--diffusion system arising from a two-level semiconductor model}
	
	\author{Michael Kniely%
		\thanks{University of Graz, Department of Mathematics and Scientific Computing, Heinrichstra\ss e 36, 8010 Graz, Austria, michael.kniely@uni-graz.at}}
	
	\date{August 7, 2025}
	
	\maketitle
	
	\begin{abstract}
		\noindent
		We consider a thermodynamically correct framework for electro--energy--reaction--diffusion systems, which feature a monotone entropy functional while conserving the total charge and the total energy. For these systems, we construct a relative entropy functional, which acts as a Lyapunov functional, and we also give an expression for the related entropy production functional. The main result of this note is the constructive derivation of an explicit entropy--entropy production inequality in a special situation, namely for a two-level semiconductor model of Shockley--Read--Hall type. Supposing that global bounded solutions exist, we show that the relative entropy decays exponentially along trajectories of the underlying system. As a consequence, these global solutions converge to the equilibrium at an exponential rate. 
	\end{abstract}
	
	
	
	\section{Introduction and preliminary results}
	\label{sec:intro} 
	Reaction--diffusion systems serve as the basic PDE model for a huge number of processes in natural and social sciences. Applications include, for example, chemical reactions of diffusing substances, charge carrier dynamics in semiconductors, phase separation processes, predator--prey models, opinion propagation, and many more. As opposed to kinetic models, reaction--diffusion systems take a simplified density-based perspective even though the underlying microscopic models may be discrete. More detailed modeling approaches often demand for a coupling of reaction--diffusion systems with, for instance, ODEs, Poisson's equation, or the Navier--Stokes equations. 
	
	\subsection{Electro--energy--reaction--diffusion systems}
	\label{subsec:eerds}
	In the present note, we focus on electro--energy--reaction--diffusion systems (EERDS), which describe the evolution of $I \in \mathbb N$ reactive, diffusive, and charged constituents subject to temperature effects in a bounded Lipschitz domain $\Omega \subset \mathbb R^d$, $d \in \mathbb N$. Typical situations, which we have in mind, are the dynamics of charge carriers in semiconductors and the motion of ions in biological tissues. To this end, we consider the generalized reaction--diffusion system 
	\begin{align}
		\label{eq:system}
		\begin{pmatrix}
			\dot{\boldsymbol c} \\ 
			\dot u 
		\end{pmatrix}
		= -\div 
		\begin{pmatrix}
			\boldsymbol j_{\!\boldsymbol c} \\ 
			j_u
		\end{pmatrix}
		+ 
		\begin{pmatrix}
			\boldsymbol R(\boldsymbol Z) \\ 
			-\boldsymbol q^\mathsf{T} \boldsymbol j_{\!\boldsymbol c} \nabla \psi
		\end{pmatrix}
	\end{align}
	accounting for energetic and electrostatic effects discussed below. The vector of non-negative concentrations $\boldsymbol c = (c_i)_i \in [0, \infty)^I$ and the internal energy $u \geq 0$ constitute the \emph{primal variables} $\boldsymbol Z = (\boldsymbol c, u) \in [0, \infty)^{I+1}$. We endow \eqref{eq:system} with no-flux boundary conditions $\nu {\cdot} j_{c_i} = \nu {\cdot} j_u = 0$ on $\partial \Omega$ for all $i = 1, \dotsc, I$ and with sufficiently smooth initial conditions $\boldsymbol Z_0$. Here and below, $\nu \in \mathbb R^d$ denotes the outer unit normal vector on $\partial \Omega$. Apart from the concentration flux $\boldsymbol j_{\!\boldsymbol c} \in \mathbb R^{I \times d}$ and the internal energy flux $j_u \in \mathbb R^{1 \times d}$, the vector $\boldsymbol R(\boldsymbol Z) \in \mathbb R^I$ stands for the reactions. We assume that $\boldsymbol q {\cdot} \boldsymbol R (\boldsymbol Z) = 0$ holds together with the charge vector $\boldsymbol q \in \mathbb Z^I$. This ensures that the total charge and the total energy defined in \eqref{eq:def-energy-charge} are conserved. Moreover, the scalar quantity $-\boldsymbol q^\mathsf{T} \boldsymbol j_{\!\boldsymbol c} \nabla \psi$ models the electric power of the motion of the charged species in the self-consistently generated electric field $\nabla \psi \in \mathbb R^d$. In this context, $\psi$ denotes the solution to Poisson's equation 
	\begin{align}
		\label{eq:poisson}
		-\div \big( \varepsilon \nabla \psi \big) = \boldsymbol q \cdot \boldsymbol c
	\end{align}
	with $\nu {\cdot} \nabla \psi = 0$ on $\partial \Omega$ and $\int_\Omega \psi \, \dd x = 0$. A necessary compatibility condition is $\int_\Omega \boldsymbol q {\cdot} \boldsymbol c \, \dd x = 0$. The parameter function $\varepsilon(x) \in [\underline \varepsilon, \overline \varepsilon]$ with constants $0 < \underline \varepsilon \leq \overline \varepsilon < \infty$ represents the uniformly positive permittivity of the underlying material. 
	We suppose that the concentration flux $\boldsymbol j_{\!\boldsymbol c}$ and the heat flux $j_u$ are given in the form 
	\begin{align}
		\label{eq:flux-cu} 
		\begin{pmatrix}
			\boldsymbol j_{\!\boldsymbol c} \\ 
			j_u 
		\end{pmatrix}
		\coleq -\mathbb M(\boldsymbol Z)
		\begin{pmatrix}
			\nabla \boldsymbol y - v \boldsymbol q \otimes \nabla \psi \\ 
			\nabla v
		\end{pmatrix}
	\end{align}
	with a positive semidefinite mobility matrix $\mathbb M(\boldsymbol Z) \in \mathbb R^{(I+1) \times (I+1)}$ and \emph{dual variables} $\boldsymbol W = (\boldsymbol y,v) = -\mathrm DS(\boldsymbol Z) \in \mathbb R^{I+1}$, which correspond to the negative (Fr\'echet) derivative of the concave entropy density $S(\boldsymbol Z)$. The vector $\boldsymbol y = (y_i)_i \in \mathbb R^I$ consists of the chemical potentials $y_i$, whereas $v < 0$ is the negative inverse temperature; cf.~\eqref{eq:def-temperature}. The term $-v \boldsymbol q \otimes \nabla \psi \in \mathbb R^{I \times d}$ gives rise to a drift along $\nabla \psi$ proportional to $v (\mathbb M(\boldsymbol Z) \boldsymbol q)_i$ for the $i$th substance. Note that the temperature 
	\begin{align}
		\label{eq:def-temperature}
		\theta \coleq -v^{-1} = \big( \mathrm D_u S(\boldsymbol Z) \big)^{-1} > 0
	\end{align}
	is positive by the strict monotonicity of $u \mapsto S(\boldsymbol c, u)$; cf.~\eqref{eq:def-entropy}, \ref{eq:hypo-sigma}, and \ref{eq:hypo-w_i}. 
	
	The structure of the flux in \eqref{eq:flux-cu} goes at least back to \cites{AGH02,M11}; see also Remark \ref{rem:correct-flux}. Using \eqref{eq:flux-cu}, one can reformulate \eqref{eq:system} in terms of a self-adjoint operator acting on the derivative of the entropy functional $\mathcal S(\boldsymbol Z)$ introduced in \eqref{eq:def-entropy-functional}. Assuming that a positive semidefinite matrix $\mathbb L(\boldsymbol Z) \in \mathbb R^{I \times I}$ exists satisfying $\boldsymbol R(\boldsymbol Z) = \mathbb L(\boldsymbol Z) \mathrm D_{\boldsymbol c} S(\boldsymbol Z)$ (cf.\ \cite{HHMM18,MM18}), we get 
	\begin{align}
		\label{eq:system-cu-a}
		\begin{pmatrix}
			\dot{\boldsymbol c} \\ 
			\dot u 
		\end{pmatrix}
		= \bigg( \mathbb A(\nabla \psi)^* \mathbb M(\boldsymbol Z) \mathbb A(\nabla \psi) + 
		\begin{pmatrix}
			\mathbb L(\boldsymbol Z) & 0 \\ 
			0 & 0
		\end{pmatrix}
		\!\bigg) \mathrm D\mathcal S(\boldsymbol Z), \qquad \mathbb A(\nabla \psi) \coleq \begin{pmatrix} \nabla_I & -\boldsymbol q \otimes \nabla \psi \\ 0 & \nabla \end{pmatrix}.
	\end{align} 
	More precisely, the operator-valued $(I{+}1) {\times} (I{+}1)$-matrix $\mathbb A(\nabla \psi)$ is defined via $(\mathbb A(\nabla \psi))_{i,i} \coleq \nabla$ for all $i = 1, \dotsc, I{+}1$, $(\mathbb A(\nabla \psi))_{i,I+1} \coleq -q_i \nabla \psi^{\mathsf T}$ for all $i = 1, \dotsc, I$, and $(\mathbb A(\nabla \psi))_{i,j} \coleq 0$ for all other $i,j = 1, \dots, I{+}1$. The adjoint operator $\mathbb A(\nabla \psi)^*$ satisfies $(\mathbb A(\nabla \psi)^*)_{i,i} = -\div$ for all $i = 1, \dotsc, I{+}1$, $(\mathbb A(\nabla \psi)^*)_{I+1,j} = -q_j \nabla \psi$ for all $j = 1, \dotsc, I$, and $(\mathbb A(\nabla \psi)^*)_{i,j} \coleq 0$ for all other $i,j = 1, \dots, I{+}1$. A more explicit equivalent representation of \eqref{eq:system} and \eqref{eq:system-cu-a} is 
	\begin{align}
		\label{eq:system-cu} 
		\begin{pmatrix}
			\dot{\boldsymbol c} \\ 
			\dot u 
		\end{pmatrix}
		= \div \bigg( \mathbb D(\boldsymbol Z) \, \begin{pmatrix} \nabla \boldsymbol c \\ \nabla u \end{pmatrix} + \theta^{-1} \mathbb M(\boldsymbol Z) \begin{pmatrix} \boldsymbol q \\ 0 \end{pmatrix} \otimes \nabla \psi \bigg) + \begin{pmatrix} \boldsymbol R(\boldsymbol Z) \\ -\boldsymbol q^{\mathsf T} \boldsymbol j_{\!\boldsymbol c} \nabla \psi \end{pmatrix} 
	\end{align}
	using the positive semidefinite diffusion matrix $\mathbb D(\boldsymbol Z) \coleq -\mathbb M(\boldsymbol Z) \mathrm D^2 S(\boldsymbol Z) \in \mathbb R^{(I+1) \times (I+1)}$. 
	
	We refer to \eqref{eq:system-np}--\eqref{eq:flux-np} for the specific EERDS, which we investigate in the main part of this paper, and which models the dynamics of electrons and holes in a semiconductor. 
	
	\begin{remark}[Electrostatic coupling]
		\label{rem:correct-flux}
		Motivated by the structure of the energy equation in \cite[Rem. 7.3]{AGH02}, the electrostatic drift term $\boldsymbol q \otimes \nabla (v \psi)$ and the energetic source term $\psi \boldsymbol q \cdot \div \boldsymbol j_{\!\boldsymbol c}$ were suggested in \cite{M11}. But both terms are explicitly depending on the bare potential $\psi$ and not only on the electric field $\nabla \psi$ as expected from physics. The correct versions given in \eqref{eq:system} and \eqref{eq:flux-cu} were subsequently published in \cite{M15}. 
	\end{remark}
	
	Consider the total energy functional $\mathcal E(\boldsymbol c,u)$ and the total charge functional $\mathcal Q(\boldsymbol c)$ defined for all $(\boldsymbol c, u) \in \mathrm L^2_+(\Omega)^{I+1}$ via 
	\begin{align}
		\label{eq:def-energy-charge}
		\mathcal E(\boldsymbol c,u) \coleq \int_\Omega \Big( \frac\varepsilon2|\nabla \psi_{\boldsymbol c}|^2 + u \Big) \, \dd x, \qquad 
		\mathcal Q(\boldsymbol c) \coleq \int_\Omega \boldsymbol q \cdot \boldsymbol c \; \dd x, 
	\end{align}
	where $\psi \equiv \psi_{\boldsymbol c} \in \mathrm H^1(\Omega)$ is the solution to Poisson's equation \eqref{eq:poisson} provided $\int_\Omega \boldsymbol q {\cdot} \boldsymbol c \, \dd x = 0$. In addition, let the total entropy functional be defined for all $(\boldsymbol c, u) \in \mathrm L^2_+(\Omega)^{I+1}$ via 
	\begin{align}
	\label{eq:def-entropy-functional}
	\mathcal S(\boldsymbol c, u) \coleq \int_\Omega S(\boldsymbol c, u) \, \dd x, 
	\end{align}
	where $S : [0, \infty)^{I+1} \rightarrow \mathbb R$ is a concave entropy density. A special class of entropy densities $S(\boldsymbol c, u)$ is given in \eqref{eq:def-entropy} below. The EERDS \eqref{eq:system-cu} is thermodynamically correct in the sense that $\mathcal E(\boldsymbol c, u)$ and $\mathcal Q(\boldsymbol c)$ are conserved and that $\mathcal S(\boldsymbol c, u)$ is non-decreasing as a function of time along solutions to the EERDS \eqref{eq:system}--\eqref{eq:flux-cu}. The conservation of charge and energy as well as the monotonicity of the entropy can be seen from the following formal calculations, where $\langle \cdot, \cdot \rangle$ denotes an appropriate duality pairing; see Remark \ref{rem:duality} for more details. On a formal level, we obtain 
	\begin{align}
		\frac{\mathrm d}{\mathrm dt} \mathcal E(\boldsymbol c, u) &= \bigg \langle \binom{\dot {\boldsymbol c}}{\dot u}, \mathrm D \mathcal E(\boldsymbol c, u) \bigg \rangle = \bigg \langle \binom{\dot {\boldsymbol c}}{\dot u}, \binom{\boldsymbol q \psi}{1} \bigg\rangle \nonumber \\ 
		&= \int_{\mathrm \Omega} \bigg( \mathbb M \, \mathbb A(\nabla \psi) \mathrm DS : \mathbb A(\nabla \psi) \binom{\boldsymbol q \psi}{1} + \binom{\boldsymbol R}{0} \cdot \binom{\boldsymbol q \psi}{1} \bigg) \, \dd x = 0 \label{eq:ddt-energy}
	\end{align}
	using \eqref{eq:system-cu-a}, $\mathbb A(\nabla \psi) (\boldsymbol q^{\mathsf T} \psi, 1)^{\mathsf T} = 0$,	and $\boldsymbol q {\cdot} \boldsymbol R = 0$. Analogously, one finds 
	\begin{align}
		\frac{\mathrm d}{\mathrm dt} \mathcal Q(\boldsymbol c, u) &= \bigg \langle \binom{\dot {\boldsymbol c}}{\dot u}, \mathrm D \mathcal Q(\boldsymbol c, u) \bigg \rangle = \bigg \langle \binom{\dot {\boldsymbol c}}{\dot u}, \binom{\boldsymbol q}{0} \bigg \rangle \nonumber \\ 
		&= \int_{\mathrm \Omega} \bigg( \mathbb M \, \mathbb A(\nabla \psi) \mathrm DS : \mathbb A(\nabla \psi) \binom{\boldsymbol q}{0} + \binom{\boldsymbol R}{0} \cdot \binom{\boldsymbol q}{0} \bigg) \, \dd x = 0, \label{eq:ddt-charge}
	\end{align}
	while the monotonicity of the total entropy follows from 
	\begin{align}
		\frac{\mathrm d}{\mathrm dt} \mathcal S(\boldsymbol c, u) &= \bigg \langle \binom{\dot {\boldsymbol c}}{\dot u}, \mathrm D \mathcal S(\boldsymbol c, u) \bigg \rangle \nonumber \\ 
		&= \int_{\mathrm \Omega} \bigg( \mathbb M \, \mathbb A(\nabla \psi) \mathrm DS : \mathbb A(\nabla \psi) \mathrm DS + \binom{\boldsymbol R}{0} \cdot \binom{\mathrm D_{\boldsymbol c} S}{\mathrm D_u S} \bigg) \, \dd x \geq 0. \label{eq:ddt-entropy}
	\end{align}
	Note that it suffices to demand $\mathrm D_{\boldsymbol c} S(\boldsymbol Z) {\cdot} \boldsymbol R(\boldsymbol Z) \geq 0$ for all $\boldsymbol Z \in [0, \infty)^{I+1}$ without imposing the additional relation $\boldsymbol R(\boldsymbol Z) = \mathbb L(\boldsymbol Z) \mathrm D_{\boldsymbol c} S(\boldsymbol Z)$. 
	
	\begin{remark}[Duality pairing]
	\label{rem:duality}
	The formal manipulations in \eqref{eq:ddt-energy}--\eqref{eq:ddt-charge} and \eqref{eq:h-dissip-s-dissip} below are consistent with the setting of Corollary \ref{cor:exp-conv} on the exponential equilibration of bounded global weak solutions $(n, p, u)$; cf.\ Definition \ref{def:weak-solution}. In this situation, we can use $\langle \cdot, \cdot \rangle \equiv \langle \cdot, \cdot \rangle_{(\mathrm W^{1,\infty})^*, \mathrm W^{1,\infty}}$ as $(\dot n, \dot p, \dot u)(t) \in \mathrm W^{1,\infty}(\Omega)^*$ holds for a.e.\ $t \geq 0$ and as the uniform boundedness of $n$ and $p$ arising from \ref{eq:lower-bound-theta}--\ref{eq:upper-bound-u} and discussed in Remark \ref{rem:reaction} leads to $\psi(t) \in \mathrm W^{1,\infty}(\Omega)$ for all $t \geq 0$. The situation in \eqref{eq:ddt-entropy} is more subtle since $\mathrm D_{\boldsymbol c} S$ is generally unbounded. However, the mere (and possibly infinite) non-negativity of $\frac{\mathrm d}{\mathrm dt} \mathcal S$ does not rely on any cancellations of integrable functions but only on the positive semidefiniteness of $\mathbb M$ and $\mathrm D_{\boldsymbol c} S {\cdot} \boldsymbol R \geq 0$ as supposed in \ref{eq:hypo-m} below. 
	\end{remark}

	We conclude this subsection with a brief overview of previous studies on electro--energy--reaction--diffusion systems (EERDS). 
	One of the first physically sound models of EERDS appeared in \cite{W90}, where stationary and transient systems for charge carriers in semiconductors are generalized to a temperature-dependent framework. The model proposed in \cite{W90} was partially reestablished in \cite{AGH02} starting from a free energy functional and employing the maximum-entropy principle. While \cites{AGH02,W90} and also \cites{M11,M15} only deal with modeling considerations, a first rigorous existence result was published in \cite{GH05} for the stationary case. More recently, the EERDS framework above was also used in \cites{BPZ17} to construct global weak solutions in a fluid setting with bounded concentrations. 
	
	\paragraph*{Notation.} The following notations and conventions are used throughout the paper. \vspace{-.5ex} 
	\begin{itemize}
		\item For $a,b \in \mathbb R^{k}$, we denote the Euclidean inner product by $a \cdot b \coleq \sum_{i=1}^k a_i b_i$. \vspace{-1ex} 
		\item For $A, B \in \mathbb R^{k \times l}$, we express the Frobenius inner product as $A : B \coleq \sum_{i=1}^k \sum_{j=1}^l A_{ij} B_{ij}$. \vspace{-1ex} 
		\item Vectors like $\boldsymbol c \in \mathbb R^I$ and $\boldsymbol Z \in \mathbb R^{I+1}$, which are related to the number of species $I \in \mathbb N$, are written in a bold face notation. Matrices like $\mathbb L \in \mathbb R^{I \times I}$ and $\mathbb M \in \mathbb R^{(I+1) \times (I+1)}$ are printed with a blackboard notation. \vspace{-1ex} 
		\item We use $C_u$, $C_\theta$ and $c_u$, $c_\theta$ for upper and lower bounds on $u$ and $\theta$, while $G_\sigma$, $G_w$, $K_\sigma$, $K_w$ and $g_w$, $k_\sigma$, $k_w$ are used for typically large and typically small constants related to the corresponding functions. 
	\end{itemize}
	
	\paragraph*{Outline.} In the remainder of this section, we define an appropriate relative entropy functional along with the corresponding entropy production functional, and we briefly discuss the purpose of deriving a so-called entropy--entropy production (EEP) inequality. Our main results on the EEP inequality and the exponential equilibration are collected in Theorem \ref{thm:EEP-inequality} and Corollary~\ref{cor:exp-conv}. Section \ref{sec:proof} is devoted to the corresponding proofs. More precisely, we start with the proof of Proposition \ref{prop:equilibrium-special} on the explicit representation of the equilibrium in Subsection \ref{subsec:proof-equil}. Propositions \ref{prop:relative-entropy-upper-bound} and \ref{prop:entropy-production-lower-bound} in Subsection \ref{subsec:proof-thm} establish upper and lower bounds for the relative entropy and the entropy production, respectively, by employing a quadratic relative entropy. Theorem \ref{thm:EEP-inequality} then immediately follows, while Corollary \ref{cor:exp-conv} is proven in Subsection \ref{subsec:proof-cor}.

	\subsection{Relative entropy and entropy production}
	\label{subsec:entropy}
	According to \cite{AGH02} and \cite{HKM24}, we introduce an equilibrium as a constrained maximizer of the entropy functional $\mathcal S(\boldsymbol c, u) = \int_\Omega S(\boldsymbol c, u) \, \dd x$. This is a purely thermodynamical characterization of an equilibrium related to the monotonicity of the total entropy (i.e., the second law of thermodynamics) and the conservation of charge and energy along trajectories of \eqref{eq:system}--\eqref{eq:flux-cu}. 
	
	\begin{definition}[Equilibrium]
		A non-negative state $\boldsymbol Z_\infty = (\boldsymbol c_\infty, u_\infty) \in \mathrm L^2_+(\Omega)^{I+1}$ is called an equilibrium of \eqref{eq:system}--\eqref{eq:flux-cu} if it is a maximizer of the total entropy $\mathcal S(\boldsymbol c, u)$ subject to prescribed values for the total energy $\mathcal E(\boldsymbol c, u) = E_0 > 0$ and the total charge $Q(\boldsymbol c) = Q_0 \in \mathbb R$, i.e., if 
		\begin{align*}
			\mathcal S(\boldsymbol c_\infty, u_\infty) = \sup \big\{ \mathcal S(\boldsymbol c, u) \ \big| \ (\boldsymbol c, u) \in \mathrm L^2_+(\Omega)^{I+1}, \ \mathcal E(\boldsymbol c, u) = E_0, \ Q(\boldsymbol c) = Q_0 \big\}. 
		\end{align*}
	\end{definition}
	
	The existence, uniqueness, and regularity of the equilibrium $(\boldsymbol c_\infty, u_\infty) \in \mathrm L^2_+(\Omega)^{I+1}$ has been proven in \cite{HKM24} in a more general setting. 
	The subsequent statement gives a condensed version of the results in \cite{HKM24} adapted to our situation. We stress that we need to demand $Q_0 = 0$ as the equilibrium potential $\psi_\infty \in \mathrm H^1(\Omega)$ is the solution to 
	\begin{align}
		\label{eq:poisson-equil}
		-\div \big( \varepsilon \nabla \psi_\infty \big) = \boldsymbol q \cdot \boldsymbol c_\infty
	\end{align}
	with $\nu {\cdot} \nabla \psi_\infty = 0$ on $\partial \Omega$ and $\int_\Omega \psi_\infty \, \dd x = 0$. 
	
	\begin{proposition}[Existence and properties of the equilibrium, cf.\ \cite{HKM24}]
		\label{prop:equilibrium}
		There exists a unique equilibrium $(\boldsymbol c_\infty, u_\infty) \in \mathrm L^2_+(\Omega)^{I+1}$. The equilibrium $(\boldsymbol c_\infty, u_\infty)$ is bounded, continuous, and uniformly positive, the equilibrium temperature $\theta_\infty > 0$ is constant in $\Omega$, and $\int_\Omega \boldsymbol q {\cdot} \boldsymbol c_\infty \, \dd x = 0$. 
	\end{proposition}
	
	Given the equilibrium $(\boldsymbol c_\infty, u_\infty)$, we construct a functional $\mathcal H$ with the following properties. First, $\mathcal H(\boldsymbol c, u \, | \, \boldsymbol c_\infty, u_\infty)$ should act as kind of a distance function between $(\boldsymbol c, u)$ and $(\boldsymbol c_\infty, u_\infty)$, hence, we look for a non-negative functional, which vanishes exactly at the equilibrium $(\boldsymbol c_\infty, u_\infty)$. This is achieved with the first three terms on the right-hand side of \eqref{eq:def-relative-entropy} due to the strict concavity of the entropy $S(\boldsymbol c, u)$. Secondly, $\mathcal H(\boldsymbol c, u \, | \, \boldsymbol c_\infty, u_\infty)$ should decrease monotonically along trajectories of \eqref{eq:system}--\eqref{eq:flux-cu}, which is fulfilled by $-\mathcal S(\boldsymbol c, u)$ itself. It turns out that the fourth term on the right-hand side of \eqref{eq:def-relative-entropy} compensates for the second and the third term as shown in \eqref{eq:h-dissip-s-dissip}. 
	
	\begin{definition}[Relative entropy]
		For any non-negative state $(\boldsymbol c, u) \in \mathrm L^2_+(\Omega)^{I+1}$ subject to $\int_\Omega \boldsymbol q {\cdot} \boldsymbol c \, \dd x = 0$, we define the (non-negative) relative entropy functional 
		\begin{align}
			\mathcal H(\boldsymbol c, u \, | \, \boldsymbol c_\infty, u_\infty) \coleq &-\int_\Omega S(\boldsymbol c, u) \, \dd x + \int_\Omega \mathrm D S(\boldsymbol c_\infty, u_\infty) 
			\begin{pmatrix}
				\boldsymbol c - \boldsymbol c_\infty \\ 
				u - u_\infty
			\end{pmatrix}
			\dd x \nonumber \\ 
			&+ \int_\Omega S(\boldsymbol c_\infty, u_\infty) \, \dd x + \frac{1}{2 \theta_\infty} \int_\Omega \varepsilon \big| \nabla (\psi - \psi_\infty) \big|^2 \, \dd x \label{eq:def-relative-entropy}
		\end{align}
		with the electrostatic potential $\psi \in \mathrm H^1(\Omega)$ being the solution to \eqref{eq:poisson}. 
	\end{definition}
	
	We now formally check that the relative entropy $\mathcal H(\boldsymbol c, u \, | \, \boldsymbol c_\infty, u_\infty)$ is a Lyapunov functional for the EERDS \eqref{eq:system}--\eqref{eq:flux-cu}. By the theory of Lagrange multipliers, there exist constants $\eta, \kappa \in \mathbb R$ such that 
	\begin{align}
		\label{eq:optimality-condition}
		\mathrm D \mathcal S(\boldsymbol Z_\infty) = \eta \mathrm D \mathcal E(\boldsymbol c_\infty, u_\infty) + \kappa \mathrm D \mathcal Q(\boldsymbol c_\infty) = \eta \binom{\boldsymbol q \psi_\infty}{1} + \kappa \binom{\boldsymbol q}{0}, 
	\end{align}
	which immediately follows from the fact that the equilibrium $\boldsymbol Z_\infty = (\boldsymbol c_\infty, u_\infty)$ is a maximizer of the total entropy under the constraints of charge and energy conservation as discussed above. With $\boldsymbol Z = (\boldsymbol c, u)$, the identities $\eta = \mathrm D_u S(\boldsymbol Z_\infty) = \theta_\infty^{-1}$, and with an appropriate duality pairing (see also Remark \ref{rem:duality}), we formally get 
	\begin{align}
		&\frac{\dd}{\dd t} \mathcal H(\boldsymbol Z | \boldsymbol Z_\infty) = - \big\langle \dot{\boldsymbol Z}, \mathrm D \mathcal S (\boldsymbol Z) \big\rangle + \big\langle \dot{\boldsymbol Z}, \mathrm D \mathcal S (\boldsymbol Z_\infty) \big\rangle + \frac{1}{\theta_\infty} \big\langle \boldsymbol q \cdot \dot{\boldsymbol c}, \psi - \psi_\infty \big\rangle \nonumber \\ 
		&= - \big\langle \dot{\boldsymbol Z}, \mathrm D \mathcal S (\boldsymbol Z) \big\rangle + \int_\Omega \mathbb A(\nabla \psi) \mathrm D \mathcal S (\boldsymbol Z_\infty) \! : \! \mathbb M(\boldsymbol Z) \mathbb A(\nabla \psi) \mathrm D S(\boldsymbol Z) \, \dd x + \frac{1}{\theta_\infty} \big\langle {-}\boldsymbol q \! \cdot \! \div \boldsymbol j_{\!\boldsymbol c}, \mathbb \psi \! - \! \psi_\infty \big\rangle \nonumber \\ 
		&= - \big\langle \dot{\boldsymbol Z}, \mathrm D \mathcal S (\boldsymbol Z) \big\rangle + \frac{1}{\theta_\infty} \int_\Omega \big( \boldsymbol q \otimes \nabla (\psi_\infty - \psi) \big) : \boldsymbol j_{\!\boldsymbol c} \; \dd x + \frac{1}{\theta_\infty} \int_\Omega \boldsymbol q^{\mathsf T} \boldsymbol j_{\!\boldsymbol c} \nabla (\psi - \psi_\infty) \, \dd x \nonumber \\ 
		&= - \frac{\dd}{\dd t} \mathcal S(\boldsymbol c, u) \leq 0. \label{eq:h-dissip-s-dissip}
	\end{align}
	Recalling the relation $\mathbb D(\boldsymbol c, u) = -\mathbb M(\boldsymbol c, u) \mathrm D^2 S(\boldsymbol c, u)$ for EERDS in Onsager form \eqref{eq:system-cu-a} and employing the hypothesis $\mathbb M(\boldsymbol c, u) = - (\mathrm D^2 S(\boldsymbol c, u))^{-1}$ from \ref{eq:hypo-m} below, we obtain $\mathbb D(\boldsymbol c, u) = \mathrm{Id}$. Thus, we proceed using $\mathrm D_u S(\boldsymbol Z) = \theta^{-1}$ and $\boldsymbol q_0 \coleq (\boldsymbol q^{\mathsf T}, 0)^{\mathsf T}$ with 
	\begin{align}
		&\frac{\dd}{\dd t} \mathcal S(\boldsymbol c, u) = \int_\Omega \Big( \mathbb A(\nabla \psi) \mathrm D \mathcal S(\boldsymbol Z) : \mathbb M(\boldsymbol Z) \mathbb A(\nabla \psi) \mathrm D \mathcal S(\boldsymbol Z) + \mathrm D_{\!\boldsymbol c} \mathcal S(\boldsymbol Z) \cdot \mathbb L(\boldsymbol Z) \mathrm D_{\!\boldsymbol c} \mathcal S(\boldsymbol Z) \Big) \, \dd x \nonumber \\ 
		&= \int_\Omega \Big( \big( \nabla \mathrm D S(\boldsymbol Z) - \theta^{-1} \boldsymbol q_0 {\otimes} \nabla \psi \big) {:} \big( {-}\nabla \boldsymbol Z + \theta^{-1} \big( \mathrm D^2 S(\boldsymbol Z) \big)^{-1} \boldsymbol q_0 {\otimes} \nabla \psi \big) + \mathrm D_{\!\boldsymbol c} S(\boldsymbol Z) \! \cdot \! \boldsymbol R(\boldsymbol Z) \Big) \, \dd x \nonumber \\ 
		&= - \int_\Omega \nabla \boldsymbol Z : \mathrm D^2 S(\boldsymbol Z) \nabla \boldsymbol Z \, \dd x + 2 \int_\Omega \theta^{-1} \boldsymbol q {\otimes} \nabla \psi : \nabla \boldsymbol c \; \dd x \nonumber \\ 
		&\color{white}= \color{black}- \int_\Omega \theta^{-2} \boldsymbol q_0 {\otimes} \nabla \psi : \big( \mathrm D^2 S(\boldsymbol Z) \big)^{-1} \boldsymbol q_0 {\otimes} \nabla \psi \, \dd x + \int_\Omega \mathrm D_{\!\boldsymbol c} S(\boldsymbol Z) \cdot \boldsymbol R(\boldsymbol Z) \, \dd x. \label{eq:ddt-s} 
	\end{align}
	
	\begin{remark}[Integrability]
		\label{rem:integrability}
		The integrands appearing in \eqref{eq:ddt-s} are generally not integrable as the total entropy may increase in a non-continuous manner, for instance, when phase transitions occur. In the situation of Corollary \ref{cor:exp-conv}, one can show that the four integrals in the last two lines of \eqref{eq:ddt-s} are either real or $+\infty$. Hence, $\frac{\dd}{\dd t} \mathcal S$ is well-defined in $[0,\infty]$ (cf.\ \eqref{eq:entropy-production}--\eqref{eq:entropy-production-law}). 
	\end{remark}
	
	We now specify the entropy function $S(\boldsymbol c, u)$ in such a way that it is concave w.r.t.\ $(\boldsymbol c, u)$, containing the Boltzmann entropy for the concentrations $\boldsymbol c$, and coupling $\boldsymbol c$ and $u$ in a nontrivial fashion. According to \cites{FHKM22,MM18}, we set 
	\begin{align}
		\label{eq:def-entropy}
		S(\boldsymbol c, u) &\coleq \hat \sigma(u) + \sum_{i=1}^I \Big( c_i \log w_i(u) - \lambda(c_i) \Big) 
	\end{align}
	with the thermal entropy $\hat \sigma : (0, \infty) \rightarrow \mathbb R$, (energetic) equilibrium densities $w_i : [0, \infty) \rightarrow (0, \infty)$, and the Boltzmann function $\lambda : [0, \infty) \rightarrow [0, \infty)$, $s \mapsto s \log s - s + 1$. With the assumption on the mobility matrix $\mathbb M( \boldsymbol c, u)$ used above, we impose the following general hypotheses on our model: 
	
	\begin{enumerate}[label={\rm(M\arabic*)},leftmargin=2\parindent]
		\item\label{eq:hypo-sigma} $\hat \sigma \in \mathrm C^2 \big( (0, \infty) \big)$ is strictly increasing and uniformly concave on bounded intervals with $\hat \sigma'(u) \rightarrow \infty$ for $u \rightarrow 0^+$ and $\hat \sigma'(u) \rightarrow 0$ for $u \rightarrow \infty$. 
		\item\label{eq:hypo-w_i} $w_i \in \mathrm C^2 \big( (0, \infty) \big) \cap \mathrm C \big( [0, \infty) \big)$ is strictly increasing and uniformly concave on bounded intervals with $w_i(0) > 0$ for all $i = 1, \dotsc, I$. 
		\item\label{eq:hypo-m} $\mathbb M(\boldsymbol c, u) \coleq - \big(\mathrm D^2 S(\boldsymbol c, u) \big)^{-1}$ and $\mathrm D_{\!\boldsymbol c} S(\boldsymbol c, u) \cdot \boldsymbol R(\boldsymbol c, u) \geq 0$ for all $(\boldsymbol c, u) \in [0, \infty)^{I+1}$. 
	\end{enumerate}
	
	\begin{remark}[Entropy, mobility, and reactions]
		As in \cite{FHKM22}, we allow in particular for $\hat \sigma(u) = a \log u$ or $\hat \sigma(u) = a u^\alpha$ with $a > 0$ and $\alpha \in (0, 1)$. But in contrast to \ref{eq:hypo-w_i}, it was assumed in \cite{FHKM22} that each $w_i$ is only concave and non-decreasing. The present condition on $w_i$ is just a mild restriction, which we state for simplicity and which is only needed in the proof of Corollary \ref{cor:exp-conv}. It does not exclude typical choices for $w_i$; see also Remark \ref{rem:cond-w}. The choice for the mobility matrix in \ref{eq:hypo-m} is frequently used in the literature to simplify the resulting models; see, e.g., \cites{HHMM18,MM18}. In real-world applications, however, it is necessary to choose $\mathbb M(\boldsymbol c, u)$ according to the problem under consideration. The so-called entropy inequality $\mathrm D_{\!\boldsymbol c} S(\boldsymbol c, u) \cdot \boldsymbol R(\boldsymbol c, u) \geq 0$ is a condition on admissible reactions $\boldsymbol R(\boldsymbol c, u)$, which holds in particular for reversible mass-action reactions $\boldsymbol R(\boldsymbol Z)$; see also \cite{MM18} and Remark \ref{rem:reaction}. 
	\end{remark}
		
	One can now derive explicit formulas for the chemical potentials $\DD_{c_i} S(\boldsymbol c, u)$ and the inverse temperature $\DD_u S(\boldsymbol c, u)$. Note that $\theta^{-1} = \DD_u S(\boldsymbol c, u) > 0$ holds due to \ref{eq:hypo-sigma}--\ref{eq:hypo-w_i}. With the entropy \eqref{eq:def-entropy}, we obtain 
	\begin{align}
		\label{eq:d-s}
		\mathrm D_{c_i} S(\boldsymbol c, u) = - \log \frac{c_i}{w_i(u)}, \qquad \qquad \mathrm D_u S(\boldsymbol c, u) = \hat \sigma'(u) + \sum_{i=1}^I c_i \frac{w_i'(u)}{w_i(u)}, 
	\end{align}
	The first identity also explains the name \emph{(energetic) equilibrium concentration} for $w_i(u)$ since $(w_i(u))_i$ is the unrestricted maximizer of $\boldsymbol c \mapsto \mathcal S(\boldsymbol c, u)$, i.e., the equilibrium concentration vector for a given internal energy density $u : \Omega \rightarrow [0, \infty)$ ignoring the charge and energy constraints. 
	Moreover, we derive the sparse Hessian as well as the full inverse Hessian (cf.\ \cite[Rem.\ 2.5]{MM18}) 
	\begin{align*}
		\mathrm D^2 S (\boldsymbol c, u) = \! 
		\begin{pmatrix}
			\ddots & & & \vdots \\ 
			& -\tfrac{1}{c_i} & & \tfrac{w'_i}{w_i} \\ 
			& & \ddots & \vdots \\ 
			\dots & \tfrac{w'_j}{w_j} & \dots & \mathrm D^2_u S 
		\end{pmatrix}\! , \quad 
		\big( \mathrm D^2 S (\boldsymbol c, u) \big)^{-1} \! = \! 
		\begin{pmatrix}
		& \vdots \\ 
		\Big({-}\delta_{ij} c_i {-} \tfrac{c_i c_j}{\gamma} \tfrac{w_i' w_j'}{w_i w_j} \Big)_{ij} & -\tfrac{c_i}{\gamma} \tfrac{w_i'}{w_i} \\ 
		& \vdots \\ 
		\dots \quad {-}\tfrac{c_j}{\gamma} \tfrac{w_j'}{w_j} \quad \dots & -\tfrac{1}{\gamma} 
		\end{pmatrix}\! ,
	\end{align*}
	where $\delta_{ij}$ is the Kronecker delta and 
	\begin{align*}
		\mathrm D^2_u S (\boldsymbol c,u) = \ &\hat \sigma''(u) + \sum_{i=1}^I c_i
		\frac{w''_i(u)}{w_i(u)}-\sum_{i=1}^I c_i
		\bigg(\frac{w'_i(u)}{w_i(u)}\bigg)^2 < 0, \\ 
		\gamma (\boldsymbol c, u) \coleq \ &{-}\hat \sigma''(u) - \sum_{i=1}^I c_i
		\frac{w''_i(u)}{w_i(u)} > 0. 
	\end{align*}
	
	We are now prepared to calculate the temporal derivative of the entropy functional $\mathcal S( \boldsymbol c, u)$ assuming that $(\boldsymbol c, u)$ is a solution to the EERDS \eqref{eq:system}--\eqref{eq:flux-cu}. To this end, we introduce the entropy production functional $\mathcal P( \boldsymbol c, u)$, which will play an essential role subsequently. 
	
	\begin{definition}[Entropy production]
		For any non-negative and weakly differentiable state $(\boldsymbol c, u) \in \mathrm L^2_+(\Omega)^{I+1}$ subject to $\int_\Omega \boldsymbol q {\cdot} \boldsymbol c \, \dd x = 0$, we formally define the (non-negative) entropy production functional 
		\begin{align}
			&\mathcal P (\boldsymbol c, u) \coleq \int_\Omega \mathrm D_{\!\boldsymbol c} S(\boldsymbol Z) \cdot \boldsymbol R(\boldsymbol Z) \, \dd x + \sum_{i=1}^I \int_\Omega \frac{1}{c_i} \bigg| \nabla c_i - c_i \frac{w_i'(u)}{w_i(u)} \nabla u + \frac{q_i c_i}{\theta} \nabla \psi \bigg|^2 \dd x \label{eq:entropy-production} \\ 
			&\!\! + \! \int_\Omega \bigg(\!{-}\hat \sigma''(u) {-} \sum_{i=1}^I c_i \frac{w_i''(u)}{w_i(u)} \bigg) \bigg| \nabla u + \bigg(\!{-}\hat \sigma''(u) {-} \sum_{i=1}^I c_i \frac{w_i''(u)}{w_i(u)} \bigg)^{\!-1} \frac{1}{\theta} \bigg( \sum_{i=1}^I q_i c_i \frac{w_i'(u)}{w_i(u)} \bigg) \nabla \psi \bigg|^2 \dd x, \nonumber
		\end{align}
		where the electrostatic potential $\psi \in \mathrm H^1(\Omega)$ is the solution to \eqref{eq:poisson}. 
	\end{definition}
		
	Note that the first integral in \eqref{eq:entropy-production} is non-negative by \ref{eq:hypo-m}. The second integral is obviously non-negative, and the third integral is non-negative due to \ref{eq:hypo-sigma}--\ref{eq:hypo-w_i}. 	
	From the last two lines of \eqref{eq:ddt-s} and the properties of the entropy in \eqref{eq:def-entropy}, we conclude that the entropy production law 
	\begin{align}
		\label{eq:entropy-production-law}
		\frac{\dd}{\dd t} \mathcal S(\boldsymbol c, u) = \mathcal P (\boldsymbol c, u)
	\end{align}
	holds along solution trajectories of \eqref{eq:system}--\eqref{eq:flux-cu}. While identity \eqref{eq:entropy-production-law} holds in general only on a formal level, it does hold as an identity in $[0, \infty]$ in the setting of Corollary \ref{cor:exp-conv}; see also Remark \ref{rem:integrability}. 
	
	It is interesting to notice that \eqref{eq:entropy-production} gives rise to a weighted integral bound on the inverse temperature gradient. To see this, we reformulate the last line of \eqref{eq:entropy-production} as 
	\begin{align*}
		&\int_\Omega \frac{1}{\gamma} \bigg| \gamma \nabla u + \frac{1}{\theta} \bigg( \sum_{i=1}^I q_i c_i \frac{w_i'(u)}{w_i(u)} \bigg) \nabla \psi \bigg|^2 \, \dd x \\ 
		&= \int_\Omega \frac{1}{\gamma} \bigg| {-}\mathrm D^2_u S(\boldsymbol c, u) \nabla u - \sum_{i=1}^I \frac{w_i'(u)}{w_i(u)} \nabla c_i + \sum_{i=1}^I \frac{w_i'(u)}{w_i(u)} \Big( \nabla c_i - c_i \frac{w_i'(u)}{w_i(u)} \nabla u + \frac{q_i c_i}{\theta} \nabla \psi \Big) \bigg|^2 \, \dd x \\ 
		&\geq \int_\Omega \frac{1}{2\gamma} \Big| \nabla \frac{1}{\theta} \Big|^2 \, \dd x - \frac{I}{\gamma} \sum_{i=1}^I \int_\Omega \Big( \frac{w_i'(u)}{w_i(u)} \Big)^2 \bigg| \nabla c_i - c_i \frac{w_i'(u)}{w_i(u)} \nabla u + \frac{q_i c_i}{\theta} \nabla \psi \bigg|^2 \, \dd x,
	\end{align*}
	where we employed $\theta^{-1} = \mathrm D_u S(\boldsymbol c, u)$ in the last step. In accordance with \ref{eq:hypo-w} below, we use that the equilibrium concentrations satisfy $(w_i')^2 \leq - g_w w_i'' w_i$ for all $i = 1, \dotsc, I$ with a constant $g_w \in (0, 1/2)$. 
	Hence, we deduce $\gamma^{-1} c_i (w_i'(u)/w_i(u))^2 \leq g_w$ and 
	\begin{align}
		\label{eq:bound-nabla-1-theta}
		\int_\Omega \bigg({-}\hat \sigma''(u) - \sum_{i=1}^I c_i \frac{w_i''(u)}{w_i(u)} \bigg)^{-1} \Big| \nabla \frac{1}{\theta} \Big|^2 \, \dd x \leq \max\{ 2, I \} \mathcal P(\boldsymbol c, u).
	\end{align}

	\subsection{Entropy--entropy production inequality}
	\label{subsec:eep}
	Our approach to proving the exponential equilibration of the EERDS \eqref{eq:system}--\eqref{eq:flux-cu} builds on the so-called entropy method, which aims to derive an entropy--entropy production (EEP) inequality relating the relative entropy functional $\mathcal H(\boldsymbol c, u \, | \, \boldsymbol c_\infty, u_\infty)$ and the entropy production functional $\mathcal P (\boldsymbol c, u)$. More precisely, we seek an explicit constant $C_\mathrm{EEP} > 0$ such that 
	\begin{align}
		\label{eq:EEP-inequality-general}
		\mathcal H \big(\boldsymbol c, u \, | \, \boldsymbol c_\infty, u_\infty \big) \leq C_\mathrm{EEP} \mathcal P (\boldsymbol c, u)
	\end{align}
	holds for all states $(\boldsymbol c, u) \in \mathrm L^2_+(\Omega)^{I+1}$ satisfying $\int_\Omega \boldsymbol q {\cdot} \boldsymbol c \, \dd x = 0$. This is a functional inequality in the sense that it does not require $(\boldsymbol c, u)$ to be a solution to the EERDS \eqref{eq:system}--\eqref{eq:flux-cu}. The EEP estimate \eqref{eq:EEP-inequality-general} can be seen as a quantified and stronger version of the second law of thermodynamics stating $\mathcal P (\boldsymbol c, u) \geq 0$. In fact, the mere non-negativity of $\mathcal P (\boldsymbol c, u)$ only ensures the monotone decay of $\mathcal H(\boldsymbol c, u \, | \, \boldsymbol c_\infty, u_\infty)$ along trajectories of \eqref{eq:system}--\eqref{eq:flux-cu}. However, \eqref{eq:EEP-inequality-general} together with an appropriate Gronwall estimate entails an exponential decay of the relative entropy in the form 
	\begin{align}
		\label{eq:exp-decay-general}
		\mathcal H \big(\boldsymbol c(t), u(t) \, | \, \boldsymbol c_\infty, u_\infty \big) \leq \mathcal H \big(\boldsymbol c_0, u_0 \, | \, \boldsymbol c_\infty, u_\infty \big) \ee^{-C_\mathrm{EEP}^{-1} t}
	\end{align}
	for all $t \geq 0$, where $(\boldsymbol c, u)$ is a trajectory of \eqref{eq:system}--\eqref{eq:flux-cu} with initial state $(\boldsymbol c_0, u_0) \in \mathrm L^2_+(\Omega)^{I+1}$. The actual convergence of a trajectory $(\boldsymbol c, u)$ to the equilibrium $(\boldsymbol c_\infty, u_\infty)$ at an exponential rate follows from a Csisz\'ar--Kullback--Pinsker estimate $\int_\Omega \big(f \log (f/g) - f + g \big) \dd x \geq C_\mathrm{CKP} \| f - g \|_{\mathrm L^1}^2$ for non-negative measurable functions $f, g : \Omega \rightarrow \mathbb R$, where $C_\mathrm{CKP} > 0$ depends on the $\mathrm L^1(\Omega)$ norm of $f$ and $g$. Assuming that appropriate bounds on $\boldsymbol c$ and $u$ are available, we obtain 
	\begin{align}
		\label{eq:exp-conv-general}
		\big\| \big(\boldsymbol c(t), u(t) \big) - (\boldsymbol c_\infty, u_\infty) \big\|_{\mathrm L^1(\Omega)}^2 + \| \psi(t) - \psi_\infty \|_{\mathrm H^1(\Omega)}^2 \leq C \ee^{-C_\mathrm{EEP}^{-1} t} 
	\end{align}
	for all $t \geq 0$, where $C > 0$ depends on $C_\mathrm{CKP}$, the equilibrium $(\boldsymbol c_\infty, u_\infty)$, the initial datum $(\boldsymbol c_0, u_0)$, and various model parameters. 
	
	The key challenge of the entropy method is clearly the derivation of the EEP inequality \eqref{eq:EEP-inequality-general}. Some of the first contributions proving exponential equilibration of electro--reaction--diffusion systems are contained in \cites{GGH96,GH97}, where the authors employ non-constructive proofs to derive functional estimates similar to \eqref{eq:EEP-inequality-general}. A constructive proof of an EEP estimate for a pure electro--diffusion system on the whole space $\mathbb R^d$ is provided in \cite{AMT00}, while similar results are established in \cites{DF06,DFFM08} for two elementary pure reaction--diffusion systems on bounded domains. We also mention the papers \cites{DFT17,FT17,FT18}, where explicit and constructive EEP inequalities for general reaction--diffusion systems modeling large classes of chemical reaction networks are derived. To our knowledge, the first explicit and constructive EEP estimate for an electro--reaction--diffusion system is presented in \cite{GG96}, where a two-level semiconductor model of Shockley--Read--Hall type is considered. This result has been transferred to a setting with self-consistent and external potentials in \cite{FK18}. Another semiconductor model involving an additional energy level for immobile electrons is investigated in \cite{FK20} and \cite{FK21}, where EEP inequalities are obtained without and with a coupling to Poisson's equation, respectively. 
	
	The interest in EEP estimates for non-isothermal reaction--diffusion systems started, on the one hand, with the work \cite{HHMM18}, where the authors introduce a gradient flow structure similar to our situation for the nonlinear reaction--diffusion model in \cite{DFFM08} generalized to the energy-dependent situation. On the other hand, an extension of the non-constructive entropy method in \cite{GGH96} to energy--reaction--diffusion systems is provided in \cite{MM18}. In this article, the authors also extend an entropy method based on a convexification argument introduced in \cite{MHM15} to the non-isothermal case. A comparison of the different kinds of isothermal entropy methods is presented in \cite{M17}.

	\section{Main results}
	\label{sec:results}
	Taking the complexity of $\mathcal P(\boldsymbol c, u)$ in \eqref{eq:entropy-production} into account, we subsequently focus on a specific two-level system for electrons and holes in a semiconductor, which can be seen as a thermodynamically consistent generalization of the model, which was already studied in \cite{FK18}. More precisely, we write $n \equiv c_1$ and $p \equiv c_2$ for the concentrations of negatively charged electrons and positively charged holes corresponding to $q_1 \coleq -1$ and $q_2 \coleq 1$. 
	
	Before stating our main result on the EEP estimate, we discuss the necessary hypotheses on top of our standing assumptions \ref{eq:hypo-sigma}--\ref{eq:hypo-m}. We suppose that the (energetic) equilibrium concentrations $w(u) \coleq w_1(u) \equiv w_2(u)$ coincide and that the following hypotheses are in place: 
	
		\begin{enumerate}[label={\rm(W\arabic*)},leftmargin=2\parindent]
			\item\label{eq:hypo-w} There exists a constant $g_w \in (0, 1/2)$ satisfying $(w')^2 \leq - g_w w'' w$ on $(0, \infty)$. 
			\item\label{eq:hypo-top-w} There exists a constant $G_w > 0$ satisfying $- w'' w \leq G_w (w')^2$ on $(0, \infty)$. 
			\item\label{eq:hypo-sigma-w} For all $c > 0$, there exists some $G_\sigma(c) > 0$ such that $- \hat \sigma'' w \leq G_\sigma(c) \hat \sigma' w'$ holds on $[c, \infty)$. 
		\end{enumerate}

	\begin{remark}[Energetic equilibrium densities]
		\label{rem:cond-w}
		Condition \ref{eq:hypo-w} was used in similar situations in \cites{MM18,FHKM22} without the restriction on $g_w$. It is, for example, fulfilled for $w(u) \coleq b (1+u)^\beta$ with arbitrary $b > 0$ and $\beta \in (0, 1/3)$ if one takes $1/2 > g_w \geq \beta / (1-\beta)$. 
		In the same situation, condition \ref{eq:hypo-top-w} is satisfied for all $G_w \geq (1-\beta)/\beta$. 
		Concerning \ref{eq:hypo-sigma-w}, we refer to Remark \ref{rem:hypo-sigma-w}. 
	\end{remark}
	
	We also restrict the class of admissible states $(n, p, u) \in \mathrm L^2_+(\Omega)^3$ to a subclass, where the following bounds on the inverse temperature $\theta^{-1}$ and the internal energy $u$ are valid: 
	
		\begin{enumerate}[label={\rm(C\arabic*)},leftmargin=2\parindent]
			\item\label{eq:lower-bound-theta} There exists a constant $c_\theta > 0$ satisfying $\theta(x) \geq c_\theta$ on $\Omega$. 
			\item\label{eq:upper-bound-u} There exists a constant $C_u > 0$ satisfying $u(x) \leq C_u$ on $\Omega$. 
		\end{enumerate}

	\begin{remark}[$\mathrm L^\infty$ bounds]
		Note that \ref{eq:lower-bound-theta}--\ref{eq:upper-bound-u} impose explicit bounds on the $\mathrm L^\infty$ norm of $u \geq 0$ and the related dual variable $\mathrm D_u S = \theta^{-1} > 0$. We further obtain the lower bound 
		\begin{align}
			\label{eq:lower-bound-u}
			u(x) \geq c_u \coleq (\hat \sigma')^{-1} \big( c_\theta^{-1} \big)
		\end{align}
		for all $x \in \Omega$ from the expression $\theta^{-1} = \hat \sigma'(u) + (n+p) w'(u)/w(u)$. The bounds $u(x) \geq c_u$, $u(x) \leq C_u$, and $\theta(x) \geq c_\theta$ are needed in \eqref{eq:condition-eep-inequality-initial}, \eqref{eq:poincare-u}, and \eqref{eq:theta-bound-needed}, respectively, when deriving a lower bound for the entropy production. Besides, we have $\theta^{-1} \geq \hat \sigma'(C_u)$, i.e., $\theta(x) \leq C_\theta \coleq \hat \sigma'(C_u)^{-1}$. 
	\end{remark}
	
	To arrive at an EEP inequality in the form \eqref{eq:EEP-inequality-general} with an explicit constant $C_\mathrm{EEP} > 0$, we need to specify the reactions $R_n(n, p, u)$ and $R_p(n, p, u)$. Having the generation and recombination of electron--hole pairs in mind, we consider mass-action kinetics for $R \coleq R_n \equiv R_p$: 
	
		\begin{enumerate}[label={\rm(R)},leftmargin=2\parindent]
			\item\label{eq:reaction-rate} $R(n, p, u) \coleq F(n, p, u) \big( w(u)^2 - np \big)$ with $F(n, p, u) \geq c_F > 0$. 
		\end{enumerate}
	
	\begin{remark}[Reactions]
		\label{rem:reaction}
		The reaction in \ref{eq:reaction-rate} satisfies the entropy inequality in \ref{eq:hypo-m}, namely $(\DD_n S + \DD_p S) R \geq 0$ (cf.\ \eqref{eq:entropy-reaction}). In particular, the setting in \ref{eq:reaction-rate} allows for uniformly positive reaction rates of Shockley--Read--Hall type 
		\begin{align*}
			F(n, p, u) = \big( k_1(u) + k_2(u) n + k_3(u) p \big)^{-1}
		\end{align*}
		with energy-dependent coefficients $k_i(u)$; see, e.g., \cites{H52,SR52}. Indeed, the formula for the inverse temperature $\theta^{-1} = \hat \sigma'(u) + (n+p) w'(u)/w(u)$ and the hypotheses \ref{eq:lower-bound-theta}--\ref{eq:upper-bound-u} give rise to uniform upper bounds $n, p \leq w(C_u)/(c_\theta w'(C_u))$ and, hence, to a uniform lower bound on $F(n, p, u)$. A similar situation was treated in \cite{FK18} in the isothermal case. We stress that the actual derivation of the EEP estimate \eqref{eq:EEP-inequality} does not rely on the boundedness of $n$ and $p$. In the case of the Shockley--Read--Hall model, the upper bound on $n$ and $p$ enters the EEP constant $C_1 C_2$ in Theorem \ref{thm:EEP-inequality} only via the lower bound $c_F$ in \ref{eq:reaction-rate}. 
	\end{remark}
	
	For the sake of completeness, we now state the explicit form of our EERDS \eqref{eq:system}--\eqref{eq:flux-cu} under the complete set of hypotheses imposed so far. In particular, we consider a two-component semiconductor model with entropy density $S(n, p, u)$ from \eqref{eq:def-entropy}, mobility $\mathbb M(\boldsymbol c, u) = - \big(\mathrm D^2 S(\boldsymbol c, u) \big)^{-1}$, and reaction term $R(n, p, u)$ as in \ref{eq:reaction-rate}. System \eqref{eq:system} then rewrites as 
	\begin{align}
		\label{eq:system-np}
		\begin{pmatrix}
			\dot n \\ 
			\dot p \\ 
			\dot u 
		\end{pmatrix}
		= -\div 
		\begin{pmatrix}
			j_n \\ 
			j_p \\ 
			j_u
		\end{pmatrix}
		+ 
		\begin{pmatrix}
			F(n, p, u) \big( w(u)^2 - np \big) \\ 
			F(n, p, u) \big( w(u)^2 - np \big) \\ 
			(j_n - j_p) \nabla \psi
		\end{pmatrix},
	\end{align}
	where the flux vectors $j_n$, $j_p$, and $j_u$ are to be understood as row vectors, and where we recall that the energetic equilibrium concentration $w(u)$ is the same for both electrons $n$ and holes $p$. Poisson's equation \eqref{eq:poisson} takes the form 
	\begin{align}
		\label{eq:poisson-np}
		-\div \big( \varepsilon \nabla \psi \big) = -n + p
	\end{align}
	with $\nu {\cdot} \nabla \psi = 0$ on $\partial \Omega$, $\int_\Omega \psi \, \dd x = 0$, and where necessarily $\int_\Omega (n-p) \, \dd x = 0$. The flux in \eqref{eq:flux-cu} is given by the following formula, where we abbreviate $\hat \sigma = \hat \sigma(u)$ and $w = w(u)$: 
	\begin{align}
		\label{eq:flux-np} 
		\begin{pmatrix}
			j_n \\ 
			j_p \\ 
			j_u
		\end{pmatrix}
		= 
		\begin{pmatrix}
			-\nabla n + \frac{n}{\theta} \Big( 1 + \frac{n-p}{-\hat \sigma'' - (n+p) \frac{w''}{w}} \big( \frac{w'}{w} \big)^2 \Big) \nabla \psi \\ 
			-\nabla p - \frac{p}{\theta} \Big( 1 - \frac{n-p}{-\hat \sigma'' - (n+p) \frac{w''}{w}} \big( \frac{w'}{w} \big)^2 \Big) \nabla \psi \\ 
			-\nabla u + \frac{1}{\theta} \frac{n-p}{-\hat \sigma'' - (n+p) \frac{w''}{w}} \frac{w'}{w} \nabla \psi 
		\end{pmatrix}.
	\end{align}
	
	Going back to \eqref{eq:d-s} and the subsequent discussion, it turns out that the equilibrium concentrations $n_\infty$ and $p_\infty$ coincide with the unrestricted entropy maximizer $w(u_\infty)$. Despite the special structure of the equilibrium established below and frequently used subsequently, we also keep the notation $n_\infty$ and $p_\infty$ whenever this seems to be more natural than $w(u_\infty)$. 
	
	\begin{proposition}[Equilibrium structure]
		\label{prop:equilibrium-special}
		The equilibrium concentrations $(n_\infty, p_\infty)$, the equilibrium energy density $u_\infty$, and the equilibrium potential $\psi_\infty$ are constant in $\Omega$ and satisfy $n_\infty = p_\infty = w(u_\infty)$, $u_\infty > 0$, and $\psi_\infty = 0$. 
	\end{proposition}

	The main result of this note is the functional EEP inequality \eqref{eq:EEP-inequality}. We actually prove that the quadratic relative entropy $\int_\Omega \big( (n-n_\infty)^2 + (p-p_\infty)^2 + (u-u_\infty)^2 \big) \, \dd x$ serves as an upper bound for the relative entropy $\mathcal H$ and as a lower bound for the entropy production $\mathcal P$. 
	
	\begin{theorem}[Entropy--entropy production inequality]
		\label{thm:EEP-inequality}
		Assume that the hypotheses \ref{eq:hypo-w}--\ref{eq:hypo-sigma-w} and \ref{eq:reaction-rate} hold true. Then, there exist positive constants $C_1, C_2 > 0$ satisfying 
		\begin{align}
			\label{eq:EEP-inequality}
			\mathcal H \big( n, p, u \, | \, n_\infty, p_\infty, u_\infty \big) \leq C_1 C_2 \mathcal P(n, p, u) 
		\end{align}
		for all non-negative states $(n, p, u) \in \mathrm L^2_+(\Omega)^3$ subject to $\int_\Omega (n-p) \, \dd x = 0$, \ref{eq:lower-bound-theta}, and \ref{eq:upper-bound-u} as well as $\nabla \sqrt{c}$, $\sqrt{c} \nabla u$, $\frac{\sqrt{c}}{\theta} \nabla \psi$, $\frac{\sqrt{c}}{\theta} \nabla \theta \in \mathrm L^2(\Omega)^d$ for $c = n, p$. Admissible choices for $C_1$ and $C_2$ are given in \eqref{eq:def-c1} and \eqref{eq:def-c2}, respectively. 
	\end{theorem}
	
	\begin{proof}
		The proof is an immediate consequence of Propositions \ref{prop:relative-entropy-upper-bound} and \ref{prop:entropy-production-lower-bound}. 
	\end{proof}
	
	We subsequently introduce a notion of a weak solution to the EERDS \eqref{eq:system-np}--\eqref{eq:flux-np}, which encompasses some minimum requirements on the densities $n$ and $p$ and on the internal energy $u$ to allow for the weak formulation in \eqref{eq:weak-formulation}. Definition \ref{def:weak-solution} only serves as a framework for the exponential equilibration result in Corollary \ref{cor:exp-conv} below. Proving the existence of any kind of global solutions will be the subject of future studies. 

	\begin{definition}[Global weak solutions]
		 \label{def:weak-solution}
		 A global weak solution to \eqref{eq:system-np}--\eqref{eq:flux-np} with initial data $(n_0, p_0, u_0) \in \mathrm L^2_+(\Omega)^3$ satisfying $\int_\Omega (n_0-p_0) \, \dd x = 0$ is a triple of non-negative functions $(n, p, u) \in \mathrm L^1_\mathrm{loc}(0, \infty; \mathrm L^2_+(\Omega))^3$ such~that 
		 \begin{itemize}
		 	\item $\dot n, \, \dot p, \, \dot u \in \mathrm L^1_\mathrm{loc} \big(0, \infty; \mathrm W^{1,\infty}(\Omega)^* \big)$, 
		 	\item $j_n, \, j_p, \, j_u \in \mathrm L^1_\mathrm{loc} \big(0, \infty; \mathrm L^1(\Omega)^d \big)$, 
		 	\item $R, \, (j_n - j_p) \! \cdot \! \nabla \psi \in \mathrm L^1_\mathrm{loc} \big(0, \infty; \mathrm L^1(\Omega) \big)$, 
		\end{itemize}
		where $j_n$, $j_p$, and $j_u$ are defined in \eqref{eq:flux-np}, $\psi(t) \in \mathrm H^1(\Omega)$ is the weak solution to \eqref{eq:poisson-np} for a.e.\ $t \geq 0$, the initial data $(n, p, u)(0) = (n_0, p_0, u_0)$ are attained in $\big( \mathrm W^{1,\infty}(\Omega)^* \big)^3$, and where 
		\begin{align}
			&\int_0^T \Big( \big\langle \dot n, \phi_n \big\rangle + \big\langle \dot p, \phi_p \big\rangle + \big\langle \dot u, \phi_u \big\rangle \Big) \, \dd t \nonumber \\ 
			&\qquad = \int_0^T \int_\Omega \Big( j_n \cdot \nabla \phi_n + j_p \cdot \nabla \phi_p + j_u \cdot \nabla \phi_u \Big) \, \dd x \, \dd t \nonumber \\ 
			&\qquad\qquad + \int_0^T \int_\Omega \Big( R(n, p, u) \big( \phi_n + \phi_p \big) + (j_n - j_p) \cdot \nabla \psi \, \phi_u \Big) \, \dd x \, \dd t \label{eq:weak-formulation}
		\end{align}
		holds for all $(\phi_n, \phi_p, \phi_u) \in \mathrm L^\infty \big( 0, T; \mathrm W^{1,\infty}(\Omega) \big)^3$ and all $T > 0$. 
	\end{definition}
	
	Applying the EEP estimate from Theorem \ref{thm:EEP-inequality} to a global weak solution $(n, p, u)$ as in Definition \ref{def:weak-solution} entails an exponential decay of the relative entropy as in \eqref{eq:exp-decay-general} by using a Gronwall argument. A subsequent application of the Csisz\'ar--Kullback--Pinsker inequality leads to the desired exponential convergence of $(n, p, u, \psi)$ to $(n_\infty, p_\infty, u_\infty, 0)$ as announced in \eqref{eq:exp-conv-general}. 

	\begin{corollary}[Exponential convergence to the equilibrium]
		\label{cor:exp-conv}
		Suppose that \ref{eq:hypo-w}--\ref{eq:hypo-sigma-w} and \ref{eq:reaction-rate} are valid. Moreover, let $(n, p, u)$ be a global weak solution to the EERDS \eqref{eq:system-np}--\eqref{eq:flux-np} with initial data $(n_0, p_0, u_0) \in \mathrm L^2_+(\Omega)^3$ subject to $\int_\Omega (n_0-p_0) \, \dd x = 0$ according to Definition \ref{def:weak-solution} satisfying \ref{eq:lower-bound-theta} and \ref{eq:upper-bound-u} uniformly for all $t \geq 0$ as well as $\nabla \sqrt{c} \big|_t$, $\sqrt{c} \nabla u \big|_t$, $\frac{\sqrt{c}}{\theta} \nabla \psi \big|_t$, $\frac{\sqrt{c}}{\theta} \nabla \theta \big|_t \in \mathrm L^2(\Omega)^d$ for $c = n, p$ and all $t \geq 0$. Then, there exists a positive constant $C_3 > 0$ such that 
		\begin{align*}
			\max \Big\{ \| n(t) - n_\infty \|_{\mathrm L^1}^2, \, \| p(t) - p_\infty \|_{\mathrm L^1}^2, \, \| u(t) - u_\infty \|_{\mathrm L^1}^2, \, \| \psi(t) \|_{\mathrm H^1}^2 \Big\} \leq C_3 \ee^{-\frac{t}{C_1 C_2}} 
		\end{align*}
		holds for all $t \geq 0$ with $C_1$ and $C_2$ as in Theorem \ref{thm:EEP-inequality}. A feasible choice is 
		\begin{align*}
			C_3 \coleq \max\Bigg\{ 2 |\Omega| \bigg( \frac{2 w(C_u)}{3 c_\theta w'(C_u)} + \frac{4 w(C_u)}{3} + \frac{w'(0)^2}{2 k_w(C_u)} \bigg), \ \frac{|\Omega|}{k_\sigma(C_u)}, \ \frac{2 (1 + C_\mathrm{P}) \theta_\infty}{\underline \varepsilon} \Bigg\} \mathcal H_0
		\end{align*}
		with $\mathcal H_0 \coleq \mathcal H \big( n_0, p_0, u_0 \, | \, n_\infty, p_\infty, u_\infty \big)$ and $k_w(C_u), \, k_\sigma(C_u) > 0$ defined in \eqref{eq:def-k-w} and \eqref{eq:def-k-sigma}. 
	\end{corollary}

	\section{Proofs of the main results}
	\label{sec:proof}
	
	This section is devoted to the proofs of the EEP inequality in Theorem \ref{thm:EEP-inequality} and the exponential equilibration of solutions to the EERDS \eqref{eq:system-np}--\eqref{eq:flux-np} in Corollary \ref{cor:exp-conv}. Before that, we prove Proposition \ref{prop:equilibrium-special} on the special structure of the equilibrium.

	\subsection{Proof of Proposition \ref{prop:equilibrium-special}}
	\label{subsec:proof-equil}
	
	The following result gives a lower bound for $\mathcal P(n, p, u)$, which allows for a more explicit control on the gradients $\nabla u$ and $\nabla \psi$. 
	\pagebreak
	\begin{lemma}
		\label{lemma:dissipative-lower-bound}
		Let the hypotheses \ref{eq:hypo-w}--\ref{eq:hypo-sigma-w} be in place. Then, we have 
		\begin{align}
			&\bigg( 1 + 2 \max \big\{ G_\sigma(c_u), \, G_w \big\}^2 \Big( g_w + \frac12 \Big) \bigg) \mathcal P(n, p, u) \geq \int_\Omega \big( \DD_n S + \DD_p S \big) R \; \dd x \nonumber \\ 
			&\qquad + \Big( \frac12 - g_w \Big) \int_\Omega \frac{\varepsilon}{\overline \varepsilon} \bigg( \frac{1}{n} \Big| \nabla n - \frac{n}{\theta} \nabla \psi \Big|^2 + \frac{1}{p} \Big| \nabla p + \frac{p}{\theta} \nabla \psi \Big|^2 \bigg) \, \dd x \nonumber \\ 
			&\qquad + \Big( \frac12 - g_w \Big) \int_\Omega \frac{\varepsilon}{\overline \varepsilon} \bigg( {-} \hat \sigma''(u) + \bigg( \Big( \frac{w'(u)}{w(u)} \Big)^2 - \frac{w''(u)}{w(u)} \bigg) (n + p) \bigg) |\nabla u|^2 \, \dd x \nonumber \\ 
			&\qquad + \Big( \frac12 - g_w \Big) \int_\Omega \frac{\varepsilon}{\overline \varepsilon} \frac{1}{\theta^2} \bigg( {-} \hat \sigma''(u) - \frac{w''(u)}{w(u)} (n + p) \bigg)^{-1} \Big( \frac{w'(u)}{w(u)} \Big)^2 (n - p)^2 \, |\nabla \psi|^2 \, \dd x \label{eq:dissipative-lower-bound}
		\end{align}
		for all non-negative states $(n, p, u) \in \mathrm L^2_+(\Omega)^3$ subject to $\int_\Omega (n-p) \, \dd x = 0$, \ref{eq:lower-bound-theta}, and $\nabla \sqrt{c}$, $\sqrt{c} \nabla u$, $\frac{\sqrt{c}}{\theta} \nabla \psi$, $\frac{\sqrt{c}}{\theta} \nabla \theta \in \mathrm L^2(\Omega)^d$ for $c = n, p$ with the positive constant $G_\sigma(c_u) > 0$ defined in \eqref{eq:condition-eep-inequality-final}. 
	\end{lemma}
	
	\begin{proof}
	We recast the entropy production \eqref{eq:entropy-production} as 
	\begin{align}
		\mathcal P(n, p, u) &= \int_\Omega \big( \DD_n S + \DD_p S \big) R \; \dd x - 2 \int_\Omega \frac{w'(u)}{w(u)} \big( \nabla n + \nabla p \big) \cdot \nabla u \, \dd x \nonumber \\ 
		&\quad + \int_\Omega \frac{1}{n} \Big| \nabla n - \frac{n}{\theta} \nabla \psi \Big|^2 \, \dd x 
		+ \int_\Omega \frac{1}{p} \Big| \nabla p + \frac{p}{\theta} \nabla \psi \Big|^2 \, \dd x \nonumber \\ 
		&\quad + \int_\Omega \bigg( {-} \hat \sigma''(u) + \bigg( \Big( \frac{w'(u)}{w(u)} \Big)^2 - \frac{w''(u)}{w(u)} \bigg) (n + p) \bigg) |\nabla u|^2 \, \dd x \nonumber \\ 
		&\quad + \int_\Omega \frac{1}{\theta^2} \bigg( {-} \hat \sigma''(u) - \frac{w''(u)}{w(u)} (n + p) \bigg)^{-1} \Big( \frac{w'(u)}{w(u)} \Big)^2 (n - p)^2 \, |\nabla \psi|^2 \, \dd x. \label{eq:entropy-production-cross-term}
	\end{align}
	The only term in \eqref{eq:entropy-production-cross-term}, which might be negative, is the second integral in the first line, and we will subsequently show how to get rid of this term. 
	An essential ingredient is our hypothesis $(w')^2 \leq - g_w w'' w$ in \ref{eq:hypo-w}, where we additionally demand $g_w \in (0, 1/2)$. The restriction $g_w < 1/2$ is needed for technical reasons below, which build on the following elementary bounds: 
	\begin{subequations}
		\label{eq:bounds-w}
		\begin{gather}
			\frac{2}{g_w + \frac12} \Big( \frac{w'}{w} \Big)^2 \leq \Big( g_w + \frac12 \Big) \Big( 1 + \frac{1}{g_w} \Big) \Big( \frac{w'}{w} \Big)^2 \stackrel{\ref{eq:hypo-w}}{\leq} \Big( g_w + \frac12 \Big) \bigg( \Big( \frac{w'}{w} \Big)^2 - \frac{w''}{w} \bigg), \label{eq:bounds-w-a} \\ 
			1 + \frac{2}{g_w + \frac12} \Big( \frac{w'}{w} \Big)^2 \frac{w}{w''} \stackrel{\ref{eq:hypo-w}}{\geq} 1 - \frac{2 g_w}{g_w + \frac12} = \frac{\frac12 - g_w}{\frac12 + g_w} \stackrel{g_w \leq \frac12}{\geq} \frac12 - g_w. \label{eq:bounds-w-b}  
		\end{gather}
	\end{subequations}
	With $\rho \coleq (g_w + 1/2)/2 \in (1/4, 1/2)$, we estimate 
	\begin{align}
		-2 \frac{w'}{w} \big( \nabla n + \nabla p \big) \cdot \nabla u \geq \ &{-} \rho \Big( \frac{w}{w''} \hat \sigma'' + n + p \Big)^{-1} | \nabla n + \nabla p |^2 \nonumber \\ 
		&{-} \frac{1}{\rho} \Big( \frac{w'}{w} \Big)^2 \Big( \frac{w}{w''} \hat \sigma'' + n + p \Big) |\nabla u|^2. \label{eq:nabla-n-nabla-p-nabla-u}
	\end{align}
	Thanks to \eqref{eq:bounds-w-a} and \eqref{eq:bounds-w-b}, the second line of \eqref{eq:nabla-n-nabla-p-nabla-u} and the third line of \eqref{eq:entropy-production-cross-term} result in 
	\begin{align*}
		\Big( \frac12 - g_w \Big) \bigg( {-} \hat \sigma''(u) + \bigg( \Big( \frac{w'(u)}{w(u)} \Big)^2 - \frac{w''(u)}{w(u)} \bigg) (n + p) \bigg) |\nabla u|^2. 
	\end{align*}
	The first line of \eqref{eq:nabla-n-nabla-p-nabla-u} is combined with a fraction of the second line of \eqref{eq:entropy-production-cross-term}, which yields 
	\begin{align}
		&\frac{2 \rho}{n} \Big| \nabla n - \frac{n}{\theta} \nabla \psi \Big|^2
		+ \frac{2 \rho}{p} \Big| \nabla p + \frac{p}{\theta} \nabla \psi \Big|^2 - \rho \Big( \frac{w}{w''} \hat \sigma'' + n + p \Big)^{-1} | \nabla n + \nabla p |^2 \nonumber \\ 
		&\quad = \frac{2 \rho}{n} |\nabla n|^2 + \frac{2 \rho}{p} |\nabla p|^2 - \rho \Big( \frac{w}{w''} \hat \sigma'' + n + p \Big)^{-1} | \nabla n + \nabla p |^2 \nonumber \\ 
		&\qquad - \frac{4 \rho}{\theta} \big( \nabla n - \nabla p \big) \cdot \nabla \psi + \frac{2 \rho (n+p)}{\theta^2} |\nabla \psi|^2, \label{eq:nabla-n-nabla-p-remainder}
	\end{align}
	where the second line of \eqref{eq:nabla-n-nabla-p-remainder} is non-negative. The previous arguments towards a lower bound on the entropy production \eqref{eq:entropy-production-cross-term} also hold true if we include the factor $\varepsilon / \overline \varepsilon \in (0, 1)$ within each integral in \eqref{eq:entropy-production-cross-term} (apart from the very first one). This is only a technical step, which is nevertheless essential to deal with the ``cross'' term in the last line of \eqref{eq:nabla-n-nabla-p-remainder}: 
	\begin{align}
		\label{eq:nabla-n-p-nabla-psi-estimate}
		-\int_\Omega \frac{4 \rho}{\overline \varepsilon \theta} \big( \nabla n - \nabla p \big) \cdot \varepsilon \nabla \psi \, \dd x = \int_\Omega \frac{4 \rho}{\overline \varepsilon \theta} (n-p)^2 \, \dd x + \frac{4 \rho}{\overline \varepsilon} \int_\Omega (n-p) \nabla \frac{1}{\theta} \cdot \varepsilon \nabla \psi \, \dd x. 
	\end{align}
	As the first integral on the right-hand side is non-negative, we are left to control the second integrand 
	\begin{align}
		\frac{4 \rho \varepsilon}{\overline \varepsilon} (n-p) \nabla \frac{1}{\theta} \cdot \nabla \psi \geq \ &{-} 2 \rho \theta^2 \Big( {-}\hat \sigma'' - \frac{w''}{w} (n+p) \Big) \Big( \frac{w'}{w} \Big)^{-2} \Big| \nabla \frac{1}{\theta} \Big|^2 \nonumber \\ 
		&{-} \frac{2 \rho}{\theta^2} \frac{\varepsilon}{\overline \varepsilon} \Big( {-}\hat \sigma'' - \frac{w''}{w} (n+p) \Big)^{-1} \Big( \frac{w'}{w} \Big)^2 (n - p)^2 \, |\nabla \psi|^2. \label{eq:nabla-1-theta-nabla-psi-estimate}
	\end{align}
	The integral over the second term is absorbed by the last line in \eqref{eq:entropy-production-cross-term}. 
	Recalling the fundamental formula $\theta^{-1} = \hat \sigma'(u) + (n+p) w'(u)/w(u)$, the integral over the first term is finally bounded by the entropy production $\mathcal P(n, p, u)$ due to \eqref{eq:bound-nabla-1-theta} provided 
	\begin{align}
		\label{eq:condition-eep-inequality-initial}
		\bigg( \frac{w'}{w} \hat \sigma' + \Big( \frac{w'}{w} \Big)^2 (n+p) \bigg)^{-1} \bigg( {-}\hat \sigma'' - \frac{w''}{w} (n+p) \bigg) \leq G(c_u) 
	\end{align}
	with a constant $G(c_u) > 0$ depending on the uniform lower bound on $u$ from \eqref{eq:lower-bound-u} being a consequence of \ref{eq:lower-bound-theta}. To this end, we recall that the equilibrium concentrations satisfy \ref{eq:hypo-top-w}, i.e., $- w'' w \leq G_w (w')^2$ with $G_w > 0$, and \ref{eq:hypo-sigma-w}, i.e., 
	\begin{align}
		\label{eq:condition-eep-inequality-final}
		- \hat \sigma''(u) \leq G_\sigma(c_u) \frac{w'(u)}{w(u)} \hat \sigma'(u) 
	\end{align}
	for all $u \geq c_u$ with a constant $G_\sigma(c_u) > 0$. This shows that \eqref{eq:condition-eep-inequality-initial} holds with the constant $G(c_u) \coleq \max \{ G_\sigma(c_u), G_w \}$. Combining all the estimates above, we arrive at \eqref{eq:dissipative-lower-bound}. 
	\end{proof}
	
	\begin{remark}
		\label{rem:hypo-sigma-w}
		We notice that the powers of $u$ appearing on both sides of \eqref{eq:condition-eep-inequality-final} essentially coincide even though \eqref{eq:condition-eep-inequality-final} does generally not hold for all $u > 0$. In particular, one can derive explicit formulas for $G_\sigma(c_u)$ if $w(u) = b (1+u)^\beta$ with $b > 0$ and $\beta \in (0, 1)$ and if $\hat \sigma(u) = a \log u$ or $\hat \sigma(u) = a u^\alpha$ with $a > 0$ and $\alpha \in (0, 1)$, namely 
		\begin{align*}
		G_\sigma(c_u) = \frac{1 + c_u}{\beta c_u} \qquad \text{or} \qquad G_\sigma(c_u) = \frac{(1-\alpha)(1 + c_u)}{\beta c_u}, 
		\end{align*}
		respectively, using $u \geq c_u$ according to \eqref{eq:lower-bound-u}. 
	\end{remark}
	
	Applying the lower bound \eqref{eq:dissipative-lower-bound} on the entropy production $\mathcal P(n, p, u)$ to the equilibrium state $(n_\infty, p_\infty, u_\infty)$ leads to the special structure of $(n_\infty, p_\infty, u_\infty, \psi_\infty)$ given in Proposition \ref{prop:equilibrium-special}. The reason for this specific situation is the assumed absence of a doping profile $D$ on the right-hand side of Poisson's equation endowed with homogeneous Neumann boundary conditions. Apart from the formal calculation of the entropy production law \eqref{eq:entropy-production-law}, these hypotheses enter the derivation of the lower entropy production bound \eqref{eq:dissipative-lower-bound} from the representation in \eqref{eq:entropy-production-cross-term} only in \eqref{eq:nabla-n-p-nabla-psi-estimate}; the remaining steps are all based on pointwise arguments. 

	\begin{proof}[Proof of Proposition \ref{prop:equilibrium-special}]
		The existence of a unique bounded, continuous, and uniformly positive equilibrium $(n_\infty, p_\infty, u_\infty)$ follows from Proposition \ref{prop:equilibrium}. As $(n_\infty, p_\infty, u_\infty)$ is a constrained maximizer of $\mathcal S(n, p, u)$, we have $\mathcal P(n_\infty, p_\infty, u_\infty) = 0$. From the third line of \eqref{eq:dissipative-lower-bound}, we thus obtain $\nabla u_\infty = 0$ in $\Omega$. We now prove that $n_\infty = p_\infty$ in $\Omega$ by assuming that $\Omega' \coleq \{x \in \Omega \, | \, n_\infty(x) \neq p_\infty(x) \} \neq \emptyset$; note that $\Omega'$ is open by the continuity of $n_\infty$ and $p_\infty$. We infer $\nabla \psi_\infty = 0$ in $\Omega'$ from the fourth line of \eqref{eq:dissipative-lower-bound} and, hence, $\nabla n_\infty = \nabla p_\infty = 0$ in $\Omega'$ from the second line of \eqref{eq:dissipative-lower-bound}. If $\Omega' = \Omega$, then $n_\infty$ and $p_\infty$ are constant in $\Omega$ and the compatibility condition for Poisson's equation \eqref{eq:poisson-equil} entails $0 = \int_\Omega (p_\infty - n_\infty) \, \dd x = |\Omega| (p_\infty - n_\infty)$, which is a contradiction. If $\Omega' \neq \Omega$, then $n_\infty$ and $p_\infty$ are constant in every connected component of $\Omega'$. But as $n_\infty = p_\infty$ holds on the boundary of each connected component inside $\Omega$ and as $n_\infty$ and $p_\infty$ are continuous, we deduce that $n_\infty = p_\infty$ also holds in $\Omega'$, which is again a contradiction. As a result, $\Omega' = \emptyset$, i.e., $n_\infty = p_\infty$ in $\Omega$. From the second line of \eqref{eq:dissipative-lower-bound}, we get $\theta_\infty^{-1} n_\infty \nabla \psi_\infty = \nabla n_\infty = \nabla p_\infty = -\theta_\infty^{-1} p_\infty \nabla \psi_\infty$. Therefore, $\nabla \psi_\infty = 0$ in $\Omega$, and $\int_\Omega \psi_\infty \, \dd x = 0$ entails $\psi_\infty = 0$ in $\Omega$ resulting in $\nabla n_\infty = \nabla p_\infty = 0$ in $\Omega$. Finally, \eqref{eq:optimality-condition} and \eqref{eq:d-s} imply $-\kappa = -\log \big( n_\infty/w(u_\infty) \big) = -\log \big( p_\infty/w(u_\infty) \big) = \kappa$, which allows to conclude that $\kappa = 0$ and $n_\infty = p_\infty = w(u_\infty)$ hold. 
	\end{proof}

	\subsection{Proof of Theorem \ref{thm:EEP-inequality}}
	\label{subsec:proof-thm}
	
	The proof of Theorem \ref{thm:EEP-inequality} is split into two parts. Proposition \ref{prop:relative-entropy-upper-bound} gives an upper bound for the relative entropy in terms of a quadratic relative entropy, while Proposition \ref{prop:entropy-production-lower-bound} shows that this quadratic relative entropy can be controlled in terms of the entropy production. 
	
	\begin{proposition}
		\label{prop:relative-entropy-upper-bound}
		There exists a positive constant $C_1 > 0$ satisfying 
		\begin{align}
		\label{eq:relative-entropy-upper-bound}
		\mathcal H(n, p, u \, | \, n_\infty, p_\infty, u_\infty) \leq C_1 \int_\Omega \Big( (n - n_\infty)^2 + (p - p_\infty)^2 + (u - u_\infty)^2 \Big) \, \dd x 
		\end{align}
		for all non-negative states $(n, p, u) \in \mathrm L^2_+(\Omega)^3$ subject to $\int_\Omega (n-p) \, \dd x = 0$ and \ref{eq:lower-bound-theta}. One can choose 
		\begin{align}
		\label{eq:def-c1}
		C_1 \coleq \max \bigg\{ \frac{2}{w(0)} + \frac{C_\mathrm{P}}{\theta_\infty \underline \varepsilon}, \ 2 \Big( \frac{2 w'(0)^2}{w(0)} + K_w \Big) + K_\sigma(c_u) \bigg\} 
		\end{align}
		with Poincar\'e's constant $C_\mathrm{P} > 0$ used in \eqref{eq:poincare-nabla-psi-n-p} and \eqref{eq:poincare}, and where the positive constants $K_\sigma(c_u)$ and $K_w$ are defined in \eqref{eq:def-K-sigma} and \eqref{eq:def-K-w}, respectively. 
	\end{proposition}
	
	\begin{proof}
		The relative entropy \eqref{eq:def-relative-entropy} (suppressing the dependence on $(n_\infty, p_\infty, u_\infty)$) rewrites as 
		\begin{align}
			\quad&\mquad \mathcal H(n, p, u) = -\int_\Omega \Big( \hat \sigma (u) + n \log w(u) - n \log n + n + p \log w(u) - p \log p + p \Big) \, \dd x \nonumber \\ 
			&+ \int_\Omega \bigg( {-} \log \Big( \frac{n_\infty}{w(u_\infty)} \Big) (n - n_\infty) - \log \Big( \frac{p_\infty}{w(u_\infty)} \Big) (p - p_\infty) \bigg) \, \dd x \nonumber \\ 
			&+ \int_\Omega \Big( \hat \sigma'(u_\infty) + \frac{w'(u_\infty)}{w(u_\infty)} (n_\infty + p_\infty) \Big) (u - u_\infty) \, \dd x + \frac{1}{2 \theta_\infty} \int_\Omega \varepsilon | \nabla \psi |^2 \, \dd x \nonumber \\ 
			&+ \int_\Omega \Big( \hat \sigma(u_\infty) + n_\infty \log w(u_\infty) - n_\infty \log n_\infty + n_\infty + p_\infty \log w(u_\infty) - p_\infty \log p_\infty + p_\infty \Big) \, \dd x \nonumber \\ 
			\quad&\mquad = \int_\Omega \bigg( n \log \Big( \frac{n}{w(u)} \Big) - n + w(u) + p \log \Big( \frac{p}{w(u)} \Big) - p + w(u) \bigg) \, \dd x \nonumber \\ 
			&- 2 \int_\Omega \Big( w(u) - w'(u_\infty) (u - u_\infty) - w(u_\infty) \Big) \, \dd x \nonumber \\ 
			&- \int_\Omega \Big( \hat \sigma(u) - \hat \sigma'(u_\infty)(u - u_\infty) - \hat \sigma (u_\infty) \Big) \, \dd x + \frac{1}{2 \theta_\infty} \int_\Omega \varepsilon | \nabla \psi |^2 \, \dd x. \label{eq:repres-rel-entropy}
		\end{align}
		A bound on $\int_\Omega \varepsilon | \nabla \psi |^2 \, \dd x$ is easily deduced from Poisson's equation \eqref{eq:poisson-np} with homogeneous Neumann boundary conditions. In fact, we derive 
		\begin{align*}
			\int_{\Omega} \varepsilon | \nabla \psi |^2 \, \dd x = -\int_\Omega (n - p) \psi \, \dd x \leq \frac{1}{2\rho} \| n - p \|_{\mathrm L^2}^2 + \frac{\rho}{2} \| \psi \|_{\mathrm L^2}^2 
		\end{align*}
		with some constant $\rho > 0$. As $\int_\Omega \psi \, \dd x = 0$, we can apply Poincar\'e's inequality in the form $\| \psi \|_{\mathrm L^2}^2 \leq C_{\mathrm P} \| \nabla \psi \|_{\mathrm L^2}^2$ with a constant $C_{\mathrm P} > 0$. Choosing $\rho \coleq \underline \varepsilon/C_{\mathrm P}$ and recalling $n_\infty = p_\infty$ entail 
		\begin{align}
			\frac{1}{2} \int_{\Omega} \varepsilon | \nabla \psi |^2 \, \dd x &\leq \frac{C_{\mathrm P}}{2 \underline \varepsilon} \int_\Omega (n - p)^2 \, \dd x \nonumber \\ 
			&\leq \frac{C_{\mathrm P}}{\underline \varepsilon} \int_\Omega \Big( (n-n_\infty)^2 + (p-p_\infty)^2 \Big) \, \dd x. \label{eq:poincare-nabla-psi-n-p}
		\end{align}	
		An analogous bound holds for $(2 \theta_\infty)^{-1} \int_\Omega \varepsilon | \nabla \psi |^2 \, \dd x$ as the equilibrium temperature $\theta_\infty > 0$ is constant in space. Recalling the elementary inequality $z \log z - z + 1 \leq (z-1)^2$ for all $z > 0$ as well as $n_\infty = w(u_\infty)$, $w(0) > 0$, and the monotonicity of $u \mapsto w(u)$, we estimate 
		\begin{align*}
			&n \log \Big( \frac{n}{w(u)} \Big) - n + w(u) \leq w(u) \Big( \frac{n}{w(u)} - 1 \Big)^2 = \frac{1}{w(u)} \big( n - w(u) \big)^2 \\ 
			&\quad \leq \frac{2}{w(u)} \Big( \big( n-n_\infty \big)^2 + \big( w(u_\infty)-w(u) \big)^2 \Big) \leq \frac{2}{w(0)} \Big( (n-n_\infty)^2 + w'(0)^2 (u-u_\infty)^2 \Big). 
		\end{align*}
		Since $\hat \sigma$ and $w$ are continuously differentiable, non-decreasing, and concave, there exist positive constants $K_\sigma(c_u), K_w > 0$ such that the following bounds hold for all $u \geq c_u$: 
		\begin{align}
			-\big( \hat \sigma(u) - \hat \sigma'(u_\infty)(u - u_\infty) - \hat \sigma (u_\infty) \big) &\leq K_\sigma (u-u_\infty)^2, \label{eq:def-K-sigma} \\ 
			-\big( w(u) - w'(u_\infty)(u - u_\infty) - w(u_\infty) \big) &\leq K_w (u-u_\infty)^2. \label{eq:def-K-w}
		\end{align}
		Collecting all estimates above, we arrive at \eqref{eq:relative-entropy-upper-bound}. 
	\end{proof}
	
	Subsequently, we will partially reuse some ideas from \cite{FK18} and \cite{GG96} to show that the right-hand side of \eqref{eq:relative-entropy-upper-bound} can be bounded from above in terms of the right-hand side of \eqref{eq:dissipative-lower-bound}. 
	
	\begin{proposition}
		\label{prop:entropy-production-lower-bound}
		Let the hypotheses \ref{eq:hypo-w}--\ref{eq:hypo-sigma-w} and \ref{eq:reaction-rate} hold true. Then, there exists a positive constant $C_2 > 0$ satisfying 
		\begin{align}
			\label{eq:entropy-production-lower-bound}
			&\int_\Omega \Big( (n-n_\infty)^2 + (p-p_\infty)^2 + (u-u_\infty)^2 \Big) \, \dd x \leq C_2 \mathcal P(n, p, u) 
		\end{align}
		for all non-negative states $(n, p, u) \in \mathrm L^2_+(\Omega)^3$ subject to $\int_\Omega (n-p) \, \dd x = 0$, \ref{eq:lower-bound-theta}, and \ref{eq:upper-bound-u} as well as $\nabla \sqrt{c}$, $\sqrt{c} \nabla u$, $\frac{\sqrt{c}}{\theta} \nabla \psi$, $\frac{\sqrt{c}}{\theta} \nabla \theta \in \mathrm L^2(\Omega)^d$ for $c = n, p$. A feasible choice is 
		\begin{align}
			\label{eq:def-c2}
			C_2 \coleq \Big( 2 + \max \big\{ 4 w'(0)^2 - 1, \, 0 \big\} \Big) \widetilde C_2, 
		\end{align}
		where we abbreviate 
		\begin{align}
			\widetilde C_2 &\coleq \bigg( \max \Big\{ 1, \, \frac{\hat \sigma'(C_u)}{4 \overline \varepsilon c_F} \Big\} + 2 \max \big\{ G_\sigma(c_u), \, G_w \big\}^2 \, \bigg) \nonumber \\ 
			&\qquad \times \frac{2 \overline \varepsilon}{1 - 2 g_w} \max\bigg\{ \frac{1}{\hat \sigma'(C_u)}, \ \frac{C_\mathrm{P}}{\underline \varepsilon} \bigg( \frac{C_{\mathrm P} w(C_u)^2}{4 \underline \varepsilon c_\theta w'(C_u)^2} - \frac{1}{\hat \sigma''(C_u)} \bigg) \bigg\} \label{eq:def-c2-tilde}
		\end{align}
		with the lower bound $c_F > 0$ for the reaction rate in \ref{eq:reaction-rate}, Poincar\'e's constant $C_\mathrm{P} > 0$ used in \eqref{eq:poincare-nabla-psi-n-p} and \eqref{eq:poincare}, and $G_\sigma(c_u) > 0$ defined in \eqref{eq:condition-eep-inequality-final}. 
	\end{proposition}
	
	\begin{proof}
	We expand the squares in the second line of \eqref{eq:dissipative-lower-bound} in Lemma \ref{lemma:dissipative-lower-bound} and integrate by parts taking Poisson's equation \eqref{eq:poisson-np} with homogeneous Neumann data into account: 
	\begin{align*}
		&\int_\Omega \frac{\varepsilon}{\overline \varepsilon} \bigg( \frac{1}{n} \Big| \nabla n - \frac{n}{\theta} \nabla \psi \Big|^2 + \frac{1}{p} \Big| \nabla p + \frac{p}{\theta} \nabla \psi \Big|^2 \bigg) \, \dd x \geq -\int_\Omega \frac{2}{\overline \varepsilon \theta} \nabla (n - p) \cdot \varepsilon \nabla \psi \, \dd x \\ 
		&\qquad = \int_\Omega \frac{2}{\overline \varepsilon \theta} (n - p)^2 \, \dd x + \frac{2}{\overline \varepsilon} \int_\Omega (n - p) \nabla \frac{1}{\theta} \cdot \varepsilon \nabla \psi \, \dd x. 
	\end{align*}
	The second integral can be controlled as in \eqref{eq:nabla-1-theta-nabla-psi-estimate}, where the $\nabla \psi$ term cancels exactly with the last line of \eqref{eq:dissipative-lower-bound}, whereas the $\nabla \theta^{-1}$ term is bounded by $\mathcal P(n, p, u)$ using \eqref{eq:bound-nabla-1-theta}. (Note that we have an additional factor $1/2-g_w$ from \eqref{eq:dissipative-lower-bound} here instead of $2\rho = g_w+1/2$ in \eqref{eq:nabla-1-theta-nabla-psi-estimate}.) Keeping only $-\hat \sigma''(u) |\nabla u|^2$ in the third line of \eqref{eq:dissipative-lower-bound}, this leads to 
	\begin{align}
		&\big( 1 + 2 \max \{ G_\sigma(c_u), G_w \}^2 \big) \mathcal P(n, p, u) \geq \int_\Omega \big( \DD_n S + \DD_p S \big) R \; \dd x \nonumber \\ 
		&\quad + \Big( \frac12 - g_w \Big) \frac{1}{\overline \varepsilon} \int_\Omega \frac{2}{\theta} (n - p)^2 \, \dd x + \Big( \frac12 - g_w \Big) \frac{\underline \varepsilon}{\overline \varepsilon} \int_\Omega \big( {-} \hat \sigma''(u) \big) |\nabla u|^2 \, \dd x. \label{eq:entropy-production-n-p-nabla-u}
	\end{align}
	From the imposed uniform upper bound $u \leq C_u$ in \ref{eq:upper-bound-u} and the energy conservation law 
	\begin{align}
		\label{eq:energy-conservation-eep-inequality}
		\int_\Omega \frac{\varepsilon}{2} |\nabla \psi|^2 \, \dd x + \int_\Omega u \, \dd x = |\Omega| u_\infty, 
	\end{align}
	we infer using Poincar\'e's inequality 
	\begin{align}
		\label{eq:poincare}
		\Big\| u - |\Omega|^{-1} \int_\Omega u \, \dd x \Big\|_{\mathrm L^2(\Omega)}^2 \leq C_{\mathrm P} \| \nabla u \|_{\mathrm L^2(\Omega)}^2
	\end{align}
	 with $C_{\mathrm P} > 0$ that 
	\begin{align}
		\label{eq:poincare-u}
		&\int_\Omega \big( {-} \hat \sigma''(u) \big) |\nabla u|^2 \, \dd x \geq {-} \hat \sigma''(C_u) \, C_{\mathrm P}^{-1} \int_\Omega \Big( u - u_\infty + \frac{1}{|\Omega|} \int_\Omega \frac{\varepsilon}{2} |\nabla \psi|^2 \, \dd x \Big)^2 \, \dd x.
	\end{align}
	One can actually choose the same constant $C_{\mathrm P}$ as in \eqref{eq:poincare-nabla-psi-n-p}. To derive a bound for $\int_\Omega (u - u_\infty)^2 \, \dd x$, we need to control $\int_\Omega \varepsilon |\nabla \psi|^2 \, \dd x$, which is achieved by recalling the first line of \eqref{eq:poincare-nabla-psi-n-p}. To be more precise, we employ Young's inequality with $\rho > 0$ specified below followed by H\"older's inequality $\| f \|_{\mathrm L^1(\Omega)}^2 \leq |\Omega| \| f \|_{\mathrm L^2(\Omega)}^2$ to get 
	\begin{align*}
		&\int_\Omega \big( {-} \hat \sigma''(u) \big) |\nabla u|^2 \, \dd x \geq {-} \hat \sigma''(C_u) \, C_{\mathrm P}^{-1} \bigg( \frac{\rho}{1+\rho} \int_\Omega ( u - u_\infty )^2 \, \dd x - \frac{\rho}{|\Omega|} \Big( \int_\Omega \frac{\varepsilon}{2} |\nabla \psi|^2 \, \dd x \Big)^2 \bigg)  \\ 
		&\quad \geq {-} \hat \sigma''(C_u) \, C_{\mathrm P}^{-1} \bigg( \frac{\rho}{1+\rho} \int_\Omega ( u - u_\infty )^2 \, \dd x - \rho \frac{C_{\mathrm P}^2}{4 \underline \varepsilon^2} \int_\Omega (n-p)^4 \, \dd x \bigg). 
	\end{align*}
	We now need to deal with $(n-p)^4$, which arises from the fact that the internal energy $u$ enters the energy constraint \eqref{eq:energy-conservation-eep-inequality} linearly, whereas the electrostatic energy depends quadratically on $\nabla \psi$. With the lower bounds $\theta \geq c_\theta$, $\theta^{-1} = \hat \sigma'(u) + (n+p) w'(u)/w(u) \geq \hat \sigma'(C_u)$, and $\theta^{-1} \geq (n+p) w'(C_u)/w(C_u)$ obtained from $u \leq C_u$, we calculate 
	\begin{align}
		&\int_\Omega \frac{2}{\theta} (n - p)^2 \, \dd x + \underline \varepsilon \int_\Omega \big( {-} \hat \sigma''(u) \big) |\nabla u|^2 \, \dd x \nonumber \\ 
		&\quad \geq \hat \sigma'(C_u) \int_\Omega (n - p)^2 \, \dd x + \int_\Omega c_\theta \frac{w'(C_u)^2}{w(C_u)^2} (n-p)^4 \, \dd x \nonumber \\ 
		&\qquad -\!\hat \sigma''(C_u) \frac{\underline \varepsilon}{C_{\mathrm P}} \bigg( \frac{\rho}{1+\rho} \int_\Omega ( u - u_\infty )^2 \, \dd x - \rho \frac{C_{\mathrm P}^2}{4 \underline \varepsilon^2} \int_\Omega (n-p)^4 \, \dd x \bigg). \label{eq:theta-bound-needed}
	\end{align}
	According to the last estimate, we choose 
	\begin{align*}
		\rho \coleq \frac{4 \underline \varepsilon c_\theta}{C_{\mathrm P}} \frac{w'(C_u)^2}{{-}\hat \sigma''(C_u) w(C_u)^2}, 
	\end{align*}
	which results in 
	\begin{align}
		&\int_\Omega \frac{2}{\theta} (n - p)^2 \, \dd x + \underline \varepsilon \int_\Omega \big( {-} \hat \sigma''(u) \big) |\nabla u|^2 \, \dd x \nonumber \\ 
		&\quad \geq \hat \sigma'(C_u) \int_\Omega \Big( \big( n-w(u) \big)^2 - 2 \big( n-w(u) \big) \big( p-w(u) \big) + \big( p-w(u) \big)^2 \Big) \, \dd x \nonumber \\ 
		&\qquad + \frac{\underline \varepsilon}{C_{\mathrm P}} \bigg( \frac{C_{\mathrm P} w(C_u)^2}{4 \underline \varepsilon c_\theta w'(C_u)^2} - \frac{1}{\hat \sigma''(C_u)} \bigg)^{-1} \int_\Omega ( u - u_\infty )^2 \, \dd x. \label{eq:n-p-nabla-u-estimate}
	\end{align}
	The bilinear expression $\big( n-w(u) \big) \big( p-w(u) \big)$ can be treated with the same technique as in \cite{FK18} by taking the reactive part of the entropy production $\big( \DD_n S + \DD_p S \big) R$ into account. To this end, we note that \eqref{eq:d-s} and the reaction term $R(n, p, u) = F(n, p, u) \big( w(u)^2 - np \big)$ with strictly positive reaction rate $F(n, p, u) \geq c_F > 0$ entail 
	\begin{align}
		\label{eq:entropy-reaction}
		\big( \DD_n S + \DD_p S \big) R = F(n, p, u) \big( np - w(u)^2 \big) \log \frac{np}{w(u)^2} \geq 0, 
	\end{align}
	which yields $n_\infty p_\infty = w(u_\infty)^2$ due to $\mathcal P(n_\infty, p_\infty, u_\infty) = 0$ independently of Proposition \ref{prop:equilibrium-special}. The key step, which originally goes back to \cite{GG96}, is as simple as elegant. Starting with the elementary inequality $(x-y) \log (x/y) \geq 4 (\sqrt x - \sqrt y)^2$ for $x \geq 0$ and $y > 0$, we deduce 
	\begin{align*}
		&\big( np - w(u)^2 \big) \log \frac{np}{w(u)^2} \geq 4 \big( \sqrt{np} - w(u) \big)^2 \\ 
		&\quad = \Big( \big(\sqrt{n} - \sqrt{w(u)}\big) \big(\sqrt{p} + \sqrt{w(u)}\big) - \big(\sqrt{n} + \sqrt{w(u)}\big) \big(\sqrt{p} - \sqrt{w(u)}\big) \Big)^2 \\ 
		&\qquad + 4 \big(\sqrt{n} - \sqrt{w(u)}\big) \big(\sqrt{p} + \sqrt{w(u)}\big) \big(\sqrt{n} + \sqrt{w(u)}\big) \big(\sqrt{p} - \sqrt{w(u)}\big) \\ 
		&\quad \geq \, 4 \big( n - w(u) \big) \big( p - w(u) \big). 
	\end{align*}
	As a consequence, we find 
	\begin{align*}
		2 \hat \sigma'(C_u) \int_\Omega \big( n - w(u) \big) \big( p - w(u) \big) \, \dd x \leq \frac{\hat \sigma'(C_u)}{2 c_F} \int_\Omega \big( \DD_n S + \DD_p S \big) R \; \dd x,
	\end{align*}
	which in combination with \eqref{eq:n-p-nabla-u-estimate} and \eqref{eq:entropy-production-n-p-nabla-u} leads to 
	\begin{align*}
		&\big( 1 + 2 \max \{ G_\sigma(c_u), G_w \}^2 \big) \mathcal P(n, p, u) + \Big( \frac{\hat \sigma'(C_u)}{4 \overline \varepsilon c_F} - 1 \Big) \int_\Omega \big( \DD_n S + \DD_p S \big) R \; \dd x \nonumber \\ 
		&\quad \geq \Big( \frac12 - g_w \Big) \frac{1}{\overline \varepsilon} \bigg[ \hat \sigma'(C_u) \int_\Omega \Big( \big( n-w(u) \big)^2 + \big( p-w(u) \big)^2 \Big) \, \dd x \\ 
		&\qquad + \frac{\underline \varepsilon}{C_{\mathrm P}} \bigg( \frac{C_{\mathrm P} w(C_u)^2}{4 \underline \varepsilon c_\theta w'(C_u)^2} - \frac{1}{\hat \sigma''(C_u)} \bigg)^{-1} \int_\Omega ( u - u_\infty )^2 \, \dd x \bigg]. 
	\end{align*}
	This proves the preliminary bound 
	\begin{align*}
		\int_\Omega \Big( \big( n-w(u) \big)^2 + \big( p-w(u) \big)^2 + (u-u_\infty)^2 \Big) \, \dd x \leq \widetilde C_2 \mathcal P(n, p, u) 
	\end{align*}
	with $\widetilde C_2$ defined in \eqref{eq:def-c2-tilde}. The assertion in \eqref{eq:entropy-production-lower-bound} now easily follows from $n_\infty = w(u_\infty)$ and 
	\begin{align*}
		\int_\Omega \big( n-n_\infty \big)^2 \, \dd x \leq 2 \int_\Omega \Big( \big( n-w(u) \big)^2 + \big( w(u)-w(u_\infty) \big)^2 \, \dd x 
	\end{align*}
	together with the concavity of $w(u)$ and an analogous argument for $p-p_\infty$. 
	\end{proof}

	\subsection{Proof of Corollary \ref{cor:exp-conv}}
	\label{subsec:proof-cor}

	Having finished the proof of the EEP inequality \eqref{eq:EEP-inequality} in Theorem \ref{thm:EEP-inequality}, we are left to establish the exponential convergence result in Corollary \ref{cor:exp-conv}. This follows essentially from a Gronwall lemma and a Csisz\'ar--Kullback--Pinsker estimate. 

	\begin{proof}[Proof of Corollary \ref{cor:exp-conv}]
		Let $(n, p, u)$ be a global weak solution to the EERDS \eqref{eq:system-np}--\eqref{eq:flux-np} with initial data $(n_0, p_0, u_0) \in \mathrm L^2_+(\Omega)^3$. The EEP estimate \eqref{eq:EEP-inequality} and a Gronwall argument give rise to 
		\begin{align*}
		\mathcal H \big( n(t), p(t), u(t) \, | \, n_\infty, p_\infty, u_\infty \big) \leq \mathcal H \big( n_0, p_0, u_0 \, | \, n_\infty, p_\infty, u_\infty \big) \exp \Big( {-}\frac{t}{C_1 C_2} \Big) 
		\end{align*}
		for all $t \geq 0$ with the constants $C_1, C_2 > 0$ from \eqref{eq:def-c1} and \eqref{eq:def-c2}. Another ingredient of the proof is the following Csisz\'ar--Kullback--Pinsker-type inequality, which holds for arbitrary measurable and non-negative functions $f, g : \Omega \rightarrow \mathbb R$ with $g$ being strictly positive: 
		\begin{align*}
			\int_\Omega \bigg( f \log \frac{f}{g} - f + g \bigg) dx \geq \frac{3}{2 \|f\|_{\mathrm L^1(\Omega)} + 4 \|g\|_{\mathrm L^1(\Omega)}} \| f - g \|_{\mathrm L^1(\Omega)}^2.
		\end{align*}
		We refer, e.g., to \cite[Lemma 4.1]{FK21} for a proof. Together with the representation of the relative entropy in \eqref{eq:repres-rel-entropy} and suppressing the dependence of $n$, $p$, $u$, and $\psi$ on the time $t$, this leads to 
		\begin{align*}
			\frac{3}{2 |\Omega| \frac{w(C_u)}{c_\theta w'(C_u)} + 4 |\Omega| w(C_u)} \big\| n - w(u) \big\|_{\mathrm L^1(\Omega)}^2 \leq \mathcal H \big( n, p, u \, | \, n_\infty, p_\infty, u_\infty \big)
		\end{align*}
		and to an analogous estimate involving $p$ using the pointwise bounds $n, p \leq w(C_u)/(c_\theta w'(C_u))$ from Remark \ref{rem:reaction}. Likewise, we find 
		\begin{align*}
			\frac{2 k_w}{w'(0)^2} \int_\Omega \big( w(u) - w(u_\infty) \big)^2 \, \dd x \leq \mathcal H \big( n, p, u \, | \, n_\infty, p_\infty, u_\infty \big), 
		\end{align*}
		where the uniform concavity of $w(u)$ on $[0, C_u]$ yields a positive constant $k_w(C_u) > 0$ satisfying 
		\begin{align}
			\label{eq:def-k-w}
			-\big( w(u) - w'(u_\infty)(u - u_\infty) - w(u_\infty) \big) \geq k_w (u-u_\infty)^2 
		\end{align}
		for all $u \in [0, C_u]$, and where we estimated $(w(u) - w(u_\infty))^2 \leq w'(0)^2 (u-u_\infty)^2$ for all $u \geq 0$. 
		The following bound now easily follows from $n_\infty = w(u_\infty)$ and $\|f\|_{\mathrm L^1(\Omega)}^2 \leq |\Omega| \|f\|_{\mathrm L^2(\Omega)}^2$: 
		\begin{align*}
			\| n - n_\infty \|_{\mathrm L^1(\Omega)}^2 &\leq 2 \Big( \big\| n - w(u) \big\|_{\mathrm L^1(\Omega)}^2 + \big\| w(u) - n_\infty \big\|_{\mathrm L^1(\Omega)}^2 \Big) \\ 
			&\leq 2 |\Omega| \bigg( \frac{2 \frac{w(C_u)}{c_\theta w'(C_u)} + 4 w(C_u)}{3} + \frac{w'(0)^2}{2 k_w} \bigg) \mathcal H \big( n, p, u \, | \, n_\infty, p_\infty, u_\infty \big). 
		\end{align*}
		A lower bound for the relative entropy in terms of the internal energy $u$ is obtained in the form 
		\begin{align*}
			\frac{k_\sigma}{|\Omega|} \| u - u_\infty \|_{\mathrm L^1(\Omega)}^2 \leq k_\sigma \int_\Omega ( u - u_\infty )^2 \, \dd x \leq \mathcal H \big( n, p, u \, | \, n_\infty, p_\infty, u_\infty \big), 
		\end{align*}
		where $k_\sigma(C_u) > 0$ is a positive constant such that 
		\begin{align}
			\label{eq:def-k-sigma}
			-\big( \hat \sigma(u) - \hat \sigma'(u_\infty)(u - u_\infty) - \hat \sigma (u_\infty) \big) \geq k_\sigma (u-u_\infty)^2
		\end{align}
		holds for all $u \in (0, C_u]$ recalling the uniform concavity of $\hat \sigma(u)$ on $(0, C_u]$. Finally, we arrive at 
		\begin{align*}
			\frac{\underline \varepsilon}{2 (1 + C_\mathrm{P}) \theta_\infty} \| \psi \|_{\mathrm H^1}^2 \leq \frac{\underline \varepsilon}{2 \theta_\infty} \int_\Omega | \nabla \psi |^2 \, \dd x \leq \mathcal H \big( n, p, u \, | \, n_\infty, p_\infty, u_\infty \big), 
		\end{align*}
		which ensures the desired bound on the electrostatic potential, where $C_\mathrm{P} > 0$ is Poincar\'e's constant as employed in \eqref{eq:poincare-nabla-psi-n-p} and \eqref{eq:poincare}. 
	\end{proof}

	\paragraph*{Acknowledgments.} 
	This research was funded in whole by the Austrian Science Fund (FWF) 10.55776/J4604. 
	For the purpose of Open Access, the author has applied a CC BY public copyright license to any Author Accepted Manuscript (AAM) version arising from this submission. 
	The author would like to thank Klemens Fellner, Katharina Hopf, and Alexander Mielke for inspiring and helpful discussions.

\end{document}